\pdfoutput=1
\documentclass[depthtwo,bib,microtype]{nicestyle}    
    
\title{Champ: A Cherednik algebra Magma package} 
\author{Ulrich Thiel}
\bibliography{references}

\newcommand{\champ}{\textsc{Champ}}

\usepackage{multirow}

\usepackage{textcomp}

\begin{document}  
\thispagestyle{empty}

\urlstyle{same}

\begin{abstract}
We present a computer algebra package based on \textsc{Magma} for performing computations in rational Cherednik algebras at arbitrary parameters and in Verma modules for restricted rational Cherednik algebras. Part of this package is a new general Las Vegas algorithm for computing the head and the constituents of a module with simple head in characteristic zero which we develop here theoretically. This algorithm is very successful when applied to Verma modules for restricted rational Cherednik algebras and it allows us to answer several questions posed by Gordon in some specific cases.  We could determine the decomposition matrices of the Verma modules, the graded $G$-module structure of the simple modules, and the Calogero–Moser families of the generic restricted rational Cherednik algebra for around half of the exceptional complex reflection groups. In this way we could also confirm Martino's conjecture for several exceptional complex reflection groups.
\end{abstract}

\blfootnote{Date: Jan 19, 2015 (first version: March 26, 2014)} \blfootnote{\textsc{Ulrich Thiel}, Universität Stuttgart, Fachbereich Mathematik, Institut für Algebra und Zahlentheorie, Lehrstuhl für Algebra, Pfaffenwaldring 57, 70569 Stuttgart, Germany.} \blfootnote{Email: \code{thiel@mathematik.uni-stuttgart.de}}

\thispagestyle{empty}
\maketitle
\thispagestyle{empty}

\vspace{-22pt}
\section*{Introduction}
\fancyhead[LO]{\footnotesize\textsc{Ulrich Thiel}}
\fancyhead[RE]{\footnotesize\textsc{\Title}}
         
\begin{parani}
Based on the computer algebra system \textsc{Magma} \cite{Magma} we developed a package, called \textsc{Champ}, which provides an environment for performing computations in rational Cherednik algebras as introduced by Etingof--Ginzburg \cite{EG-Symplectic-reflection-algebras} and in Verma modules for restricted rational Cherednik algebras as introduced by Gordon \cite{Gor-Baby-verma}. It is freely available at \url{http://thielul.github.io/CHAMP/} and consists of around 16,000 lines of code at the moment. It is designed to be highly flexible so that it is possible to work with arbitrary parameters (including indeterminates of a rational function field and thus covering the generic setting), with arbitrary reflection groups over arbitrary fields (including fields of positive characteristic as long as all reflections are diagonalizable), and with arbitrary realizations of the irreducible representations of the reflection groups (see \S\ref{champ}). The development of this package was motivated by questions posed by Gordon \cite[\S7]{Gor-Baby-verma} (see \S\ref{gordon_problems}) and by Martino's conjecture \cite{Mar-CM-Rouqier-partition} (see \S\ref{martino_conjecture_subsection}) which relates Calogero–Moser families with Rouquier families coming from Hecke algebras (see \cite{BroKim02-Familles-de-cara}, \cite{MalRou-Familles-de-caracteres-de-0}, and \cite{Chl09-Blocks-and-famil}). For \textit{exceptional} complex reflection groups almost nothing was known about this. Using the theoretical methods developed here and their implementation in \textsc{Champ} we could make significant progress (see \S\ref{further_results} for a summary and the ancillary document of this article for all results).

In \S\ref{cherednik} we introduce rational Cherednik algebras over general base rings and deduce the PBW theorem in this generality by using properties of rewrite systems. We discuss an efficient algorithm for performing computations in these algebras, i.e., for expressing products in the PBW basis. This has been implemented in \textsc{Champ} and allows us for example to explicitly compute Poisson brackets which have a variety of applications (see \cite{Bonnafe.C;Rouquier.R13Cellules-de-Calogero}). In \S\ref{restricted_cherednik} we discuss an efficient algorithm for computing Verma modules for restricted rational Cherednik algebras. This allows us to construct and handle Verma modules even of dimension around 3,000 in \textsc{Champ}. As there is so far no algorithm capable of decomposing such high-dimensional modules over a field of characteristic zero, it is one of the central advances in this article that we theoretically develop a very general strategy for doing this (see paragraphs \S\ref{ffspecs} to \S\ref{heads}). We say ``strategy'' here as our theory yields a so-called Las Vegas algorithm, meaning that it does not have to be successful but if it is we get the correct result. We have implemented this algorithm—with a lot of technical extensions—in \champ. Our idea is to use \textit{finite field specializations} (which are compositions of decomposition morphisms in the sense of Geck–Rouquier \cite{GR-Centers-Simple-Hecke}) to transport the modules to an algebra over a finite field (see \S\ref{ffspecs}), then apply the \textsc{MeatAxe} \cite{Hol-Meataxe}, and use a method for \textit{reconstructing} the head of the original module—the latter is the essential part of our approach (see \S\ref{modfinder}) and culminates in an algorithm we call \textsc{ModFinder}.  To apply this algorithm to Verma modules for restricted rational Cherednik algebras we first have to ensure the existence of ``integral structures'' of these algebras. This is an interesting theoretical problem which has not been considered before. In \S\ref{ffspecs} we develop some theory around this problem and present an algorithmic partial (but for us sufficient) solution. Despite the uncertainty in the success of this algorithm it turned out to be extremely efficient and successful for Verma modules. Namely, we are able to compute for all the exceptional complex reflection groups
\[
\rG_4, \rG_5, \rG_6, \rG_7, \rG_8, \rG_9, \rG_{10}, \rG_{12}, \rG_{13}, \rG_{14}, \rG_{15}, \rG_{16}, \rG_{20}, \rG_{22}, \rG_{23}=\rH_3, \rG_{24}
\]
the decomposition matrices of the Verma modules, the structure of the simple modules as graded $G$-modules, and the Calogero–Moser families of the associated generic restricted rational Cherednik algebra—and thus the answers to Gordon's questions in these cases.\footnote{The reader should check the website \url{http://thielul.github.io/CHAMP/} and \cite{Thi-CHAMP:-A-Cherednik-A2014} for further results obtained after publication of this article.}  Nothing was known about this before. Moreover, we confirm in this way the generic part of Martino's conjecture for these groups. As \textsc{Champ} was designed to handle arbitrary parameters (including generic points of subschemes) we are also able to do the same for \textit{all} parameters for the groups $\rG_4$, $\rG_{12}$, $\rG_{13}$, $\rG_{20}$, $\rG_{22}$, and $\rG_{23}=\rH_3$, and confirm the complete form of Martino's conjecture in these cases. For the groups $\rG_{4}$, $\rG_6$, $\rG_8$, $\rG_{13}$, $\rG_{14}$, and $\rG_{20}$ we furthermore give an explicit description of the ``exceptional locus'', which could not be determined so far. It coincides precisely with the union of Chlouveraki's essential hyperplanes of cyclotomic Hecke algebras \cite{Chl09-Blocks-and-famil}—except for $\rG_8$ where we surprisingly have one additional ``exceptional'' hyperplane (this was discovered before by Bonnafé using entirely different methods). 

All results are listed explicitly in tabular form in the ancillary document of this article and are easily accessible from within \champ\ for future work (see \S\ref{champ_db}). In \S\ref{further_results} we summarize them along with some observations. %

We hope that our package and our results will enable us to better understand  problems about rational Cherednik algebras like the precise connection between Calogero–Moser families and Rouquier families, and the recent Calogero–Moser cell conjecture by Bonnafé–Rouquier \cite{Bonnafe.C;Rouquier.R13Cellules-de-Calogero}. We expect that our method for computing the heads and decomposition matrices of Verma modules can be applied to many more examples outside of rational Cherednik algebras.
\end{parani}

\begin{ack}
I would like to thank Claus Fieker for showing me some tricks in \textsc{Magma} which led to improvements of \champ. Furthermore, I would like to thank Gunter Malle for several comments on a preliminary version of this article. I am also thankful to the referee of LMS J.\ Comput.\ Math.\ for several remarks. I was partially supported by the \textit{DFG Schwerpunktprogramm Darstellungstheorie 1388}.
\end{ack}

\tableofcontents

\section{Computing in rational Cherednik algebras} \label{cherednik}

\begin{parani}
We start by reviewing rational Cherednik algebras (see also \cite{EG-Symplectic-reflection-algebras}, \cite{Gor-Baby-verma}, \cite{Bonnafe.C;Rouquier.R13Cellules-de-Calogero}, and \cite{Thiel.UOn-restricted-ration}) in this paragraph and explain how they can be treated computationally. Instead of the complex numbers as base rings we consider a very general setup here to be able to treat generic parameters algebraically and to introduce analogous problems with modular reflection groups. We argue that the PBW theorem follows in this generality from the fact that there exists a terminating confluent rewrite system for rational Cherednik algebras. As a by-product, this formalizes an algorithm for computing in these algebras and proves its correctness.
\end{parani}

\subsection{Rational Cherednik algebras}

Throughout, we fix a field $K$ and a finite reflection group $\Gamma \dopgleich (G,V)$ over $K$. This means that $G$ is a non-trivial finite group, $V$ is a finite-dimensional faithful $KG$-module, and $G$ is generated by the set $\mrm{Ref}_\Gamma$ of elements $s \in G$ which act as reflections on $V$, i.e., those elements whose fixed space $\rH_s \dopgleich \Ker(\id_V-s)$ is of codimension one. We denote the action of $g \in G$ on $v \in V$ by $\,^g v$. For $s \in \mrm{Ref}_\Gamma$ we denote by $\alpha_s^\vee$ a \word{root} of $s$, i.e., a non-zero element of $\Im(\id_V-s)$, and by $\alpha_s$ we denote a \word{coroot} of $s$, i.e., an element of $V^*$ whose kernel is equal to $\rH_s$. Both roots and coroots of reflections are unique up to scalars and our constructions will not depend on their choice.

We assume that all reflections in $G$ are diagonalizable. This is  equivalent to $\langle \alpha_s^\vee,\alpha_s \rangle \neq 0$ for all $s \in \mrm{Ref}_\Gamma$, where $\langle \cdot, \cdot \rangle$ is the canonical pairing between $V$ and $V^*$. As all reflections in $\Gamma$ are of finite order, this is certainly satisfied if $\Gamma$ is \word{non-modular}, i.e., the characteristic of $K$ is coprime to the order of $G$. In the modular case the general orthogonal groups in their natural representation in case $K$ is of characteristic not equal to $2$, the symmetric group $\rS_n$ in the representation attached to the partition $(n-1,1)$ in case $K$ is of characteristic not equal to $2$, and some modular reductions of exceptional complex reflection groups satisfy this property for example (see \cite{Thiel.UOn-restricted-ration}). 
In addition to $\Gamma$ we furthermore fix a commutative $K$-algebra $R$, an element $t \in R$, and a map $c:\sC_\Gamma \rarr R$ from the set $\sC_\Gamma$ of conjugacy classes of reflections of $\Gamma$ to $R$. The \word{rational Cherednik algebra} of $\Gamma$ in $(t,c)$ is defined as the quotient $\rH_{t,c}$ of $R \langle V \oplus V^* \rangle \rtimes RG$ by the ideal $\rI_{t,c}$ generated by the relations
\begin{align} 
& \lbrack x,x' \rbrack = 0 \quad \tn{for all } x,x' \in V^*  \;, \label{cherednik_relation_1} \\ 
& \lbrack y,y' \rbrack = 0 \quad \tn{for all } y,y' \in V \;, \label{cherednik_relation_2}
\end{align}
and
\begin{equation} \label{cherednik_relation_3}
\lbrack y,x \rbrack = t\langle y,x \rangle + \sum_{s \in \mrm{Ref}_\Gamma} (y,x)_s c(s) s \quad \tn{for all } x \in V^*, y \in V \;,
\end{equation}
where
\begin{equation}
(y,x)_s \dopgleich \frac{ \langle y,\alpha_s \rangle \langle \alpha_s^\vee,x \rangle}{\langle \alpha_s^\vee,\alpha_s \rangle} \in K\;.
\end{equation}
Here,  we denote by $R \langle V \rangle$ the tensor algebra of $V^*$ over $R$ and by $R \lbrack V \rbrack$ we denote the symmetric algebra of $V^*$, i.e., the quotient of $R \langle V \rangle$ by the ideal generated by the elements $xx'-x'x$ for $x,x' \in V^*$. Furthermore, $R \langle V \oplus V^* \rangle \rtimes RG$ denotes the semi-direct product of the tensor algebra of $V^* \oplus V$ over $R$ with the group algebra over $R$. As we assumed that all reflections are diagonalizable, we have $\langle \alpha_s^\vee,\alpha_s \rangle \neq 0$ so that the last relation is always well-defined. Note that it is also independent of the choice of the roots and coroots. 

\subsection{The PBW theorem}

Let $\biy \dopgleich (y_i)_{i=1}^n$ be a basis of $V$ with dual basis $\bix \dopgleich (x_i)_{i=1}^n$. We denote by $\rF_n$ the set of finite sequences $\alpha \dopgleich (\alpha_1,\ldots,\alpha_l)$ in $\lbrack 1,n \rbrack \dopgleich \lbrace 1,\ldots,n \rbrace$ and define for such a sequence the expression $\bix_\alpha \dopgleich \prod_{i=1}^{l} x_{\alpha_i} \in R \langle V \rangle$. Then $(\bix_\alpha)_{\alpha \in \rF_n}$ is an $R$-basis of $R \langle V \rangle$. An $R$-basis of $R \lbrack V \rbrack$ is formed by the elements $\bix^\alpha \dopgleich \prod_{i=1}^n x_i^{\alpha_i}$ with $\alpha \in \bbN^n$. The choice of a basis provides us with a natural $R$-linear section of the quotient morphism $R \langle V \rangle \twoheadrightarrow R \lbrack V \rbrack$ by mapping $\bix^\alpha \in R \lbrack V \rbrack$ to $\prod_{i=1}^n x_i^{\alpha_i} \in R \langle V \rangle$. In the same way we have a natural $R$-linear section of $R \langle V^* \rangle \twoheadrightarrow R \lbrack V^* \rbrack$. As an $R$-module the semi-direct product $R \langle V \oplus V^* \rangle \rtimes RG$ is isomorphic to $R \langle V \oplus V^* \rangle \otimes_R RG$. 
The two sections above can thus be put together to yield an $R$-linear section $s_\biy$ of the quotient morphism $R \langle V \oplus V^* \rangle \rtimes RG \twoheadrightarrow R \lbrack V \oplus V^* \rbrack \rtimes RG$. The image $N_\biy$ of $s_\biy$ is the free $R$-submodule of $R \langle V \oplus V^* \rangle \rtimes RG$ with basis $\bix^\alpha \biy^\beta g$ and we get a commutative diagram

\[
\begin{tikzcd}
\mbox{} & N_\biy \arrow[dashed, hookrightarrow]{d}\\
\mbox{} & R \langle V \oplus V^* \rangle \rtimes RG \arrow[twoheadrightarrow]{dl} \arrow[twoheadrightarrow]{dr} \\
R\lbrack V \oplus V^* \rbrack \rtimes RG \arrow[dashed, bend left]{uur}{s_\biy}[swap]{\cong} \arrow[dashed,swap]{rr}{\pi} & & \rH_{t,c} 
\end{tikzcd}
\]
where the dashed arrows are morphisms of $R$-modules only and $\pi$ is the composition of $s_\biy$ with the quotient morphism. This morphism is actually independent of the choice of $\biy$ and is called the \word{PBW morphism}.  It is clear from the relations (\ref{cherednik_relation_1}) and (\ref{cherednik_relation_2}) that $\pi$ is surjective so that the elements $\bix^\alpha \biy^\beta g$ generate $\rH_{t,c}$ as an $R$-module. The essence of the \word{PBW theorem} for rational Cherednik algebras is that $\pi$ is in fact an isomorphism (equivalently, the restriction of the quotient morphism $R \langle V \oplus V^* \rangle \rtimes RG \twoheadrightarrow \rH_{t,c}$ to $N_\biy$ is injective for one, and then any, basis $\biy$). Hence, the elements $\bix^\alpha \biy^\beta g$ with $\alpha,\beta \in \bbN^n$ form an $R$-basis of $\rH_{t,c}$. We call such a basis a \word{PBW basis}. One sometimes prefers to use that $R\lbrack V \oplus V^* \rbrack \rtimes RG$ is as an $R$-module isomorphic to $R\lbrack V \rbrack \otimes_R RG \otimes_R R \lbrack V^* \rbrack$ so that we have a triangular decomposition of $\rH_{t,c}$ and a basis of the form $\bix^\alpha g \biy^\beta$. This fact is used in \S\ref{rrca_represent_theory}. \\
The PBW theorem was originally proven by Etingof–Ginzburg \cite{EG-Symplectic-reflection-algebras} in the case $K = R = \bbC$. Their proof, however, seems to be not easily extendable to our general setting. Ram–Shepler \cite{RS-Classification-of-graded-Hecke} instead gave a proof in the same case which is formalized and extended in \cite{Thiel.UOn-restricted-ration}. The advantage of this approach is not only that it can be adapted to give a proof of the PBW theorem over general base rings but that it also provides the theoretical foundation of our computational approach to rational Cherednik algebras. To explain this let us first formalize the role of $N_\biy$ in the PBW theorem.

\subsection{Normal forms and rewrite systems}

\begin{definition}  \label{symb_Nna}\label{normal_form_algebra_def}
Let $A$ be an algebra over a commutative ring $R$ and let $I \unlhd A$ be an ideal. A \word{weak normal form} of $A/I$ is an $R$-submodule $N \subs A$ such that any element of $A$ is modulo $I$ equivalent to an element of $N$, i.e., the restriction $\pi|_N$ of the quotient morphism $\pi:A \twoheadrightarrow A/I$ to $N$ is still surjective. For $a \in A$ we call the elements in $\sN_N(a) \dopgleich \pi|_N^{-1}(\pi(a)) = \pi^{-1}(\pi(a)) \cap N$ the \word{normal forms} of $a$ with respect to $N$, and similarly we define $\sN_N(\ol{a}) \dopgleich \pi|_N^{-1}(\ol{a}) = \pi^{-1}(\ol{a}) \cap N$ for $\ol{a} \in A/I$. 
If every element of $A$ has a unique normal form with respect to $N$, i.e., $\pi|_N:N \twoheadrightarrow A/I$ is an isomorphism of $R$-modules, we say that $N$ is a \word{normal form} of $A/I$. 
\end{definition}

\begin{para}
Finding a normal form for a quotient of a (commutative) polynomial ring by an ideal is one of the central problems of computational commutative algebra and it can be solved via Gröbner bases as explained in the following example.
\end{para}

\begin{example}\label{example_groebner_basis_normal_form}
Let $A \dopgleich K\lbrack \biX \rbrack$ be the polynomial ring over a field $K$ in the variables $\biX \dopgleich (X_i)_{i=1}^n$. Let $\prec$ be a monomial order on $A$. Let $I \unlhd A$ be an ideal and let $G \dopgleich \lbrace g_1,\ldots, g_s \rbrace$ be a Gröbner basis of $I$ with respect to $\prec$, i.e., $\mrm{LT}(I) = \mrm{LT}(G)$, where $\mrm{LT}(-)$ denotes the ideal generated by the leading terms. Let 
\[
C(I) \dopgleich \lbrace \biX^\alpha \mid \alpha \in \bbN^n \tn{ and } \biX^\alpha \tn{ is not divisible by some } \mrm{LT}(g) \tn{ for } g \in G \rbrace \subs A \;.
\]
Then $N_I \dopgleich \langle C(I) \rangle_K \subs A$ is a normal form of $A/I$ (see \cite[\S1.2]{Geck.M03An-introduction-to-a}).
\end{example}

\begin{para}
We can reformulate the PBW theorem as stating that the $R$-submodule $N_\biy$ is a normal form for $\rH_{t,c} = (R \langle V \oplus V^* \rangle \rtimes RG)/\rI_{t,c}$. We will show this by proving that there exists a terminating confluent rewrite system having $N_\biy$ as the set of normal forms. To this end, let us first recall some basic notions about rewrite systems (see \cite{Bezem.M;Klop.J;Vrijer.R03Term-rewriting-syste}).
\end{para}

\begin{definition} \label{rewrite_systems}
A \word{rewrite system} is a pair $\sA \dopgleich (A, \rarr )$ consisting of a set $A$ and a binary relation $\rarr$ on $A$. We write $a \rarr b$ if $(a,b) \in \ \rarr$. This relation is called the \word{rewrite relation} of $\sA$. The reflexive-transitive closure of $\rarr$ is denoted by $\twoheadrightarrow$. An element $a \in A$ is \word{reducible} if there is some $b \in A$ with $a \neq b$ and $a \rarr b$. Otherwise it is called \word{irreducible} (or in \word{normal form}).  A \word{normal form} of an element $a \in A$ is an irreducible element $b \in A$ with $a \twoheadrightarrow b$. We denote by $\sN_\sA(a)$ the set of normal forms of $a$. The rewrite system $\sA$ is (\words{uniquely}{normalizing!uniquely}) \word{normalizing} if every element $a \in A$ has a (unique) normal form. It is called \word{terminating} if there does not exist an infinite chain $a_1 \rarr a_2 \rarr \cdots$. It is called \word{locally confluent} if
     \[
     \forall a,b,c \in A \left( c \leftarrow a \rightarrow  b \Rightarrow \exists d \in A \left(c \twoheadrightarrow d \twoheadleftarrow b \right) \right) \;.
     \]
     This condition is precisely the commutativity of the diagram
     
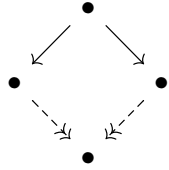
\[
     \begin{tikzcd}[column sep=small, row sep=small]
     & {\bullet} \arrow{dl} \arrow{dr} \\
      {\bullet} \arrow[dashed, twoheadrightarrow]{dr} & &  \bullet \arrow[dashed, twoheadrightarrow]{dl} \\
      & \bullet
     \end{tikzcd}
     \]
     where the vertices denote the corresponding elements of $A$ and the dashed arrows indicate the existence condition. Finally, $\sA$ is called \word{confluent} if
     \[
     \forall a,b,c \in A \left( c \twoheadleftarrow a \twoheadrightarrow  b \Rightarrow \exists d \in A \left(c \twoheadrightarrow d \twoheadleftarrow b \right) \right) \;.
     \]
 \end{definition}

\begin{para}
Very helpful for proving confluence of a rewrite system is Newman's lemma which states that a terminating rewrite system is confluent if and only if it is locally confluent (see \cite[Theorem 1.2.1]{Bezem.M;Klop.J;Vrijer.R03Term-rewriting-syste}). Let us record some further elementary facts about rewrite systems.  
\end{para}

\begin{lemma} \label{rewrite_system_properties}
    The following holds for a rewrite system $\sA \dopgleich (A,\rarr)$:
    \begin{enum_thm}
        \item \label{rewrite_system_properties:term} If $\sA$ is terminating, then $\sA$ is normalizing. 
        \item \label{rewrite_system_properties:confl} If $\sA$ is confluent, then any element of $A$ has at most one normal form.
        \item \label{rewrite_system_properties:unorm} $\sA$ is uniquely normalizing if and only if it is normalizing and confluent.
    \end{enum_thm}
\end{lemma}

\begin{proof}
    Assertions \ref{rewrite_system_properties:term} and \ref{rewrite_system_properties:confl} are easy to see. If $\sA$ is uniquely normalizing, it is normalizing by definition. To see that $\sA$ is confluent let $a,b,c \in A$ with $c \twoheadleftarrow a \twoheadrightarrow b$. Let $\wt{c}$ be a normal form of $c$ and let $\wt{b}$ be a normal form of $b$. We then have $a \twoheadrightarrow b \twoheadrightarrow \wt{b}$ and $a \twoheadrightarrow c \twoheadrightarrow \wt{c}$. Since $\wt{b}$ and $\wt{c}$ are irreducible, they are both normal forms of $a$. But then $\wt{b} = \wt{a} = \wt{c}$, where $\wt{a}$ is the unique normal form of $a$. This shows that $\sA$ is confluent. The other direction is evident.
    \end{proof}

\begin{para}
We want to establish a rewrite system on an algebra with respect to an ideal. Such a rewrite system should satisfy some natural compatibility conditions. We propose the following definition (there seems to be no established general theory yet).
\end{para}

   \begin{definition} \label{rewrite_system_for_algebra}
Let $A$ be an algebra over a commutative ring $R$ and let $I \unlhd A$ be an ideal. A \word{rewrite system} for $A/I$ is a rewrite system $\sA \dopgleich (A,\rarr)$ on $A$ satisfying the following properties:
 \begin{enum_thm}
      \item \label{rewrite_system_for_algebra:equiv} If $a \rarr b$, then $a \equiv b \modd I$ for all $a,b \in A$.
      \item \label{rewrite_system_for_algebra:norm_ra} If $a \in A$ is irreducible, also $ra$ is irreducible for all $r \in R$. %
      \item \label{rewrite_system_for_algebra:norm_ab} If $a,b \in A$ are irreducible, also $a+b$ is irreducible.%
 \end{enum_thm} 
\end{definition}

\begin{para}
We can now relate the two notions of normal forms in Definition  \ref{normal_form_algebra_def} and Definition \ref{rewrite_systems}. The following two lemmas are the key to the PBW theorem.
\end{para}

\begin{lemma} \label{rewrite_system_for_algebra_properties}
    Let $A$ be an algebra over a commutative ring $R$, let $I \unlhd A$ be an ideal and let $\sA \dopgleich (A,\rarr)$ be a rewrite system for $A/I$. The following holds:
    \begin{enum_thm}
        \item \label{rewrite_system_for_algebra_properties:equiv} If $a \twoheadrightarrow b$, then $a \equiv b \modd I$ for all $a,b \in A$.
        \item \label{rewrite_system_for_algebra_properties:norm} If $\sA$ is normalizing, then  
\[
N_\sA \dopgleich \bigcup_{a \in A} \sN_\sA(a) \subs A
\]
is a weak normal form of $A/I$ with $\sN_\sA(a) \subs \sN_{N_\sA}(a)$ for all $a \in A$. 

    \end{enum_thm}
\end{lemma}

\begin{proof}
The first assertion follows immediately from Definition \ref{rewrite_system_for_algebra}\ref{rewrite_system_for_algebra:equiv} and the fact that $\equiv$ is both reflexive and transitive. Furthermore, Definition \ref{rewrite_system_for_algebra}\ref{rewrite_system_for_algebra:norm_ra} and Definition \ref{rewrite_system_for_algebra}\ref{rewrite_system_for_algebra:norm_ab} imply that $N_\sA$ is an $R$-submodule of $A$ and it is then a weak normal form of $A/I$ due to \ref{rewrite_system_for_algebra_properties:equiv}
\end{proof}

\begin{lemma} \label{rewrite_system_for_algebra_properties2}
    Let $A$ be an algebra over a commutative ring $R$, let $I \unlhd A$ be an ideal and let $\sA \dopgleich (A,\rarr)$ be a normalizing rewrite system for $A/I$. The following are equivalent:
    \begin{enum_thm}
        \item \label{rewrite_system_for_algebra_properties2:nf} $N_\sA$ is a normal form of $A/I$.
        \item \label{rewrite_system_for_algebra_properties2:confl_ideal} $a \twoheadrightarrow 0$ for all $a \in I$.
    \end{enum_thm}
    In this case $\sA$ is uniquely normalizing and $\sN_\sA(a) = \sN_{N_\sA}(a)$ for all $a \in A$.
\end{lemma}

\begin{proof}
Suppose that $N_\sA$ is a normal form of $A/I$. Then $\sN_{N_\sA}(a)$ is a singleton for all $a \in A$. Since $\sA$ is normalizing and $\sN_\sA(a) \subs \sN_{N_\sA}(a)$, this implies that $\sN_\sA(a) = \sN_{N_\sA}(a)$ and so $\sN_\sA(a)$ is also a singleton. Hence, $\sA$ is uniquely normalizing. Moreover, if $a \in I$, then $\sN_\sA(a) = \sN_{N_\sA}(a) = \pi^{-1}(\pi(a)) \cap N_{\sA} = \pi^{-1}(0) \cap N_{\sA} = I \cap N_{\sA} = \lbrace 0 \rbrace$. Hence, $a \twoheadrightarrow 0$ for all $a \in I$. 

Now, suppose that \ref{rewrite_system_for_algebra_properties2:confl_ideal} holds.  To show that $N_\sA$ is a normal form, we show that the restriction $\pi|_{N_\sA}$ of the quotient morphism $\pi:A \twoheadrightarrow A/I$ to $N_\sA$ is injective. If $\wt{a}$ is an element of the kernel of this morphism, then $\wt{a} \in N_\sA \cap I$, so $\wt{a}$ is an irreducible element contained in $I$. But the assumption that $a \twoheadrightarrow 0$ for all $a \in I$ implies that whenever $a \in I$ is irreducible, then already $a=0$. Hence $\wt{a} = 0$ and so $N_\sA$ is a normal form of $A/I$.   
\end{proof}

\begin{remark}
If $\sA$ is normalizing and satisfies $a \twoheadrightarrow 0$ for all $a \in I$,  then it follows from Lemma \ref{rewrite_system_for_algebra_properties2} and Lemma \ref{rewrite_system_properties}\ref{rewrite_system_properties:unorm} that $\sA$ is confluent. The condition $a \twoheadrightarrow 0$ for all $a \in I$ might, however, be stronger than confluence. In other words, confluence of $\sA$ alone might not be sufficient for making $N_\sA$ into a normal form for $A/I$.
\end{remark}

\subsection{Monomial rewrite systems}

 Defining rewrite relations for $A/I$ is much more intricate than it seems at first---in particular when it comes to verifying confluence and the property $a \twoheadrightarrow 0$ for all $a \in I$. Usually, one would tend to define rewrite relations on \textit{symbolic monomials} of $A$, which we understand as symbolic concatenations of elements of $A$ symbolizing a product, and then extend these relations to \word{symbolic expressions}, i.e., symbolic monomials involving parentheses, addition and subtraction symbols. But this approach leads to the following major issue. Let $a \in A$ be an irreducible element and let $b \in A$ be a reducible element. In $A$ we have of course $a = a + b - b$ but as symbolic expressions $a$ and $a+b-b$ are distinct. Since $b$ is reducible and we extended the rewrite rules by linearity, also $a+b-b$ is reducible. This is a contradiction since in $A$ this symbolic term becomes equal to $a$ which is irreducible. Because of this one has to be very careful when defining rewrite relations for $A/I$. We can avoid this problem by defining rewrite rules on basis elements of $A$ and then extending these linearly. We formalize this in the following definition.

\begin{definition}
\label{monomial_rewrite_system}
    Let $\bia \dopgleich (a_\lambda)_{\lambda \in \Lambda}$ be an $R$-basis of $A$. In this context we call the elements $a_\lambda$ also \word{monomials} of $A$ and by \word{terms} we understand multiples $ra_\lambda$ with $r \in R \setminus \lbrace 0 \rbrace$. If $a \in A$ we say that a term $r a_\lambda$ is a \word{term of} $a$ if it occurs in the basis representation of $a$. 
    Now, suppose that $\rarr$ is a subset of $(a_\lambda)_{\lambda \in \Lambda'} \times A$ for some subset $\Lambda' \subs \Lambda$, i.e., $\rarr$ relates some monomials of $A$ with elements of $A$. We extend $\rarr$ to a relation $\rarr'$ as follows:
  \begin{enum_thm}
  \item If $a \in A$ and $ra_\lambda$ is a term of $a$ with $a_\lambda \rarr b$, then $a \rarr' a-ra_\lambda + rb$.
  \item If $a_\lambda \rarr b$ and $a_\mu = xa_\lambda y$ for some $\lambda,\mu \in \Lambda$ and $x,y \in A$, then $a_\mu \rarr' xby$. 
  \end{enum_thm}  
The first extension rule should be understood as removing the term $ra_\lambda$ from $a$ and replacing it by $rb$. The second extension rule means that we can apply rules to ``submonomials'' of monomials. We call the rules defined by $\rarr$ the \word{elementary rules} of the resulting rewrite system and call rewrite systems defined like this \word{monomial rewrite systems}.
 \end{definition}

  It is easy to see that a monomial rewrite system on an algebra $A$ satisfies Definition \ref{rewrite_system_for_algebra}\ref{rewrite_system_for_algebra:norm_ra} and Definition \ref{rewrite_system_for_algebra}\ref{rewrite_system_for_algebra:norm_ab}. So, what remains to be verified to establish it as a rewrite system for $A/I$ is Definition  \ref{rewrite_system_for_algebra}\ref{rewrite_system_for_algebra:equiv} on elementary rules (note that $I$ is a two-sided ideal). Suppose that in this case we can furthermore show that the resulting rewrite system $\sA$ for $A/I$ is terminating and that $a \twoheadrightarrow 0$ for all $a \in I$ holds. Then we know from Lemma \ref{rewrite_system_properties}\ref{rewrite_system_properties:term} that $\sA$ is normalizing and so it follows from Lemma \ref{rewrite_system_for_algebra_properties2} that $\sA$ is already uniquely normalizing. Furthermore, the module theoretic notion of normal forms in Definition \ref{normal_form_algebra_def} coincides with the rewrite system theoretic one in Definition \ref{rewrite_systems}. 

 \begin{theorem} \label{pbw_theorem_conclusion}
 Define the monomial rewrite system $\sA_{t,c,\biy}$ on $R \langle V \otimes V^* \rangle \rtimes RG$ with respect to the $R$-basis $\bix_\alpha \biy_\beta g$ by the following elementary rules:
\begin{align}
& x_j x_i \rarr x_ix_j &\tn{for } j>i \;, \quad \label{che_rewrite_rule1} \\
& y_j y_i \rarr y_iy_j & \tn{for } j>i \;, \quad \label{che_rewrite_rule2} \\
& y_i x_j \rarr x_j y_i + t\langle y_i,x_j \rangle + \sum_{s \in \mrm{Ref}_\Gamma} (y,x)_s c(s) s &  \tn{for all } i,j  \label{che_rewrite_rule3} \;. \ \hspace{4.2pt}
\end{align}
This rewrite system is terminating and satisfies $a \twoheadrightarrow 0$ for all $a \in \rI_{t,c}$. It is thus a uniquely normalizing rewrite system for $\rH_{t,c}$. 
\end{theorem}

\begin{proof}
This is a tedious but straightforward computation (see \cite[\S16]{Thiel.UOn-restricted-ration}). 
\end{proof}

It is obvious that $N_{\sA_{t,c,\biy}} = N_\biy$, and as Lemma \ref{rewrite_system_for_algebra_properties2} implies that $N_\biy$ is a normal form for $\rH_{t,c}$, this proves the PBW theorem.
  
\begin{remark}
The proof of the PBW theorem is given \cite[\S16]{Thiel.UOn-restricted-ration} by the same arguments for the much more general Drinfeld–Hecke algebras (see also \cite{RS-Classification-of-graded-Hecke}). The class of such algebras includes for example the symplectic reflection algebras by Etingof–Ginzburg \cite{EG-Symplectic-reflection-algebras}. With the straightforward adaptions of the algorithms we discuss in the next section we can thus compute in these algebras, too.
\end{remark}

 \subsection{Computing in rational Cherednik algebras}

Our approach to the PBW theorem using rewrite systems directly gives us a first  algorithm for computing in rational Cherednik algebras. As the semi-direct product is usually not supported by computer algebra systems we switch to a ``cover'' which is supported, namely the tensor algebra $R \langle \bix \cup \biy \cup \mbi{g} \rangle$, where $\mbi{g}\dopgleich (g_k)_{k=1}^r$ is a system of generators of $G$. We equip this algebra with the same rewrite rules as in Theorem \ref{pbw_theorem_conclusion} and the additional monomial rewrite rules
\begin{align}
& g_k x_i \rarr \,^{g_k} x_i g_k \quad \tn{for all } i \tn{ and } k \;, \label{che_rewrite_rule4} \\
& g_k y_i \rarr \,^{g_k} y_i g_k \quad \tn{for all } i \tn{ and } k \label{che_rewrite_rule5}\;.
\end{align}
This yields a confluent terminating rewrite system on $R \langle \bix \cup \biy \cup \mbi{g} \rangle$. It does not take care of the relations in the group, so to get PBW basis expressions we have to rewrite the ``group algebra part'' of each monomial of the normal form of an element uniquely as a word in the generators $\mbi{g}$. We can do this by choosing unique representations for every element of $G$.\\

Although straightforward, this algorithm is very inefficient as the elements in the tensor algebra can become very large and as we apply just one rule at a time. There is a much more efficient way to compute in rational Cherednik algebras.\footnote{The use of the group algebra instead of the tensor algebra was suggested and already used by Cédric Bonnafé.} Namely, the PBW theorem implies that $\rH_{t,c}$ is as an $R$-module isomorphic to the group algebra $R \lbrack V \oplus V^* \rbrack G$ of $G$ over the commutative ring $R \lbrack V \oplus V^* \rbrack$ so that we can consider $\rH_{t,c}$ as $R \lbrack V \oplus V^* \rbrack G$ with a modified multiplication. Working in $R \lbrack V \oplus V^* \rbrack G$ instead of $R \langle \bix \cup \biy \cup \mbi{g} \rangle$ is much more efficient as the commutativity of the $\bix$ and the $\biy$ is already inherent so that we do not need rewrite rules for this, and we do not have to rewrite group elements. Moreover, Lemma \ref{commutator_formula} below provides an explicit commutator formula which combines several rewrite rules and thus allows much faster computation of products. The idea for computing a product $ab$ in $\rH_{t,c} \cong R\lbrack V \oplus V^* \rbrack G$ is then to multiply each term of $a$ with each term of $b$ using the commutator formula in Lemma \ref{commutator_formula} and sum up the result. This is made precise in Algorithm \ref{alg_computing_in_cherednik}. Here, we denote by $a_g(\bix,\biy) \in R \lbrack V \oplus V^* \rbrack$ the coefficient of $g$ of an element $a \in R \lbrack V \oplus V^* \rbrack G$, so $a = \sum_{g \in G} a_g(\bix,\biy)g$. Although we have six nested loops in this algorithm it is still very efficient. In the implementation in \textsc{Champ} we also make use of a database of commutators which is updated during runtime. This leads to an additional speed-up. The most time-consuming part of the algorithm is the computation of the action of elements of $G$ on polynomials in $R \lbrack V \oplus V^* \rbrack$. 

\begin{lemma} \label{commutator_formula}
 For any $\mu \in \bbN^n$ the following relation holds in $\rH_{t,c}$:
 \begin{equation}
 \lbrack y_i, x_1^{\mu_1} \cdots x_n^{\mu_n} \rbrack = \lbrack y_i, x_1^{\mu_1} \cdots x_n^{\mu_n} \rbrack_0 + \lbrack y_i, x_1^{\mu_1} \cdots x_n^{\mu_n} \rbrack_t \;,
 \end{equation}
 where 
 \begin{equation}
\lbrack y_i, x_1^{\mu_1} \cdots x_n^{\mu_n} \rbrack_0 \dopgleich  \sum_{s \in \mrm{Ref}_\Gamma} \lbrack y_i, x_1^{\mu_1} \cdots x_n^{\mu_n} \rbrack _s s
 \end{equation}
 with
 \begin{equation}
\lbrack y_i, x_1^{\mu_1} \cdots x_n^{\mu_n} \rbrack _s \dopgleich c(s) \sum_{j=1}^n (y_i,x_j)_s x_1^{\mu_1} \cdots x_{j-1}^{\mu_{j-1}} \left( \sum_{l=0}^{\mu_j-1} x_j^l \,^s(x_j^{\mu_j-l-1}) \right) \,^s \left( x_{j+1}^{\mu_{j+1}} \cdots x_n^{\mu_n} \right) \;,
\end{equation}
 and 
 \begin{equation}
  \lbrack y_i, x_1^{\mu_1} \cdots x_n^{\mu_n} \rbrack_t \dopgleich t \sum_{j=1}^n \mu_j x_1^{\mu_1} \cdots x_{j-1}^{\mu_{j-1}} x_j^{\mu_j-1} x_{j+1}^{\mu_{j+1}} \cdots x_n^{\mu_n} \langle y_i, x_j \rangle \;.
 \end{equation}
 \end{lemma}
 
 \begin{proof}
This is a straightforward proof by induction we omit here.
 \end{proof}
 
\begin{algorithm}[htbp] 
 \caption{Computation of products in rational Cherednik algebras}
 \label{alg_computing_in_cherednik}
\KwData{Elements $a = \sum_{g \in G} a_g(\bix,\biy) g$ and $b = \sum_{h \in G} b_h(\bix,\biy) h$ of $R \lbrack V \oplus V^* \rbrack G$}
\KwResult{The product $c \dopgleich ab$ in $\rH_{t,c} \cong R \lbrack V \oplus V^* \rbrack G$}  
$c \dopgleich 0$\;
\For{$g \in G\tn{ with }a_g(\bix,\biy) \neq 0$}
{
    $d \dopgleich 0$; //this will be $\left(a_g(\bix,\biy)g\right) b$ in the end\\
    $e \dopgleich \sum_{h \in G} \,^gb_h(\bix,\biy)gh$; //$e = gb \in \rH_{t,c}$\\
    //now we compute $a_g(\bix,\biy)e = a_g(\bix,\biy)gb$\\
    \For{$t\tn{ a term of }a_g(\bix,\biy)$}
    {
        $m_t \dopgleich \tn{the monomial of }t\tn{, so }m_t=\bix^\alpha \biy^\nu\tn{ for some }\alpha,\nu \in \bbN^n$\;
        $k_t \dopgleich \tn{the coefficient of }t$\;
        $E \dopgleich e$; //this will be $\biy^\nu e$ in the end\\ 
        \For{$i \dopgleich 1\tn{ to }n$}
        {
            \For{$j \dopgleich 1\tn{ to }\nu_i$}
            {
                $l \dopgleich 0$; //this will be $y_i E$ \\
                \For{$h \in G\tn{ with }E_h(\bix,\biy) \neq 0$}
                {
                    \For{$u\tn{ a term of }E_h(\bix,\biy)$}
                    {
                        $m_u \dopgleich \tn{the monomial of }u\tn{, so }m_u=\bix^\mu \biy^\beta\tn{ for some }\mu,\beta \in \bbN^n$\;
                        $k_u \dopgleich \tn{the coefficient of }u$\;   
                        $l \dopgleich l + k_u( \bix^\mu y_i \biy^\beta + \lbrack y_i,\bix^\mu \rbrack_t \biy^\beta + \sum_{s \in \mrm{Ref}_\Gamma} \lbrack y_i,\bix^\mu \rbrack_s \,^s\biy^\beta sh)$\;
                        //the second summand above is simply the PBW expression \\//for $k_u y_i \bix^\mu \biy^\beta h =$ from from Lemma \ref{commutator_formula} \\
                    }
                }
                $E \dopgleich l$\;
            }
        }
        $d \dopgleich d + k_t \bix^\alpha E$; //$d \dopgleich d + te$\\
    }
    $c \dopgleich c + d$; //$c \dopgleich c + a_g(\bix,\biy)gb$\\
}
\Return{$c$}\;
\end{algorithm}

\subsection{Poisson brackets} \label{poisson_brackets}
One of the motivations for devising and implementing algorithms for computing in rational Cherednik algebras is that this allows us to explicitly compute Poisson brackets of central elements of $\rH_{0,c}$. We give a non-standard (but equivalent) definition of the Poisson bracket here as this is more efficient for computations (see \cite[5.4.A]{Bonnafe.C;Rouquier.R13Cellules-de-Calogero} for the usual definition). Let $\wt{R} \dopgleich D \otimes_K R$, where $D \dopgleich K \lbrack \eps \rbrack/\langle \eps^2 \rangle$ is the ring of dual numbers. We denote the image of $\eps$ in $D$ again by $\eps$. The map $c:\sC_\Gamma \rarr R$ can of course also be considered as mapping to $\wt{R}$ and so the $\wt{R}$-algebra $\rH_{\eps,c}$ is defined. By the PBW theorem we have $\rH_{\eps,c} \cong D \otimes_K \rH_{0,c}$ as $R$-modules and we have a canonical embedding $\wt{\cdot}:\rH_{0,c} \hookrightarrow \rH_{\eps,c}$ of $R$-modules. There is a canonical surjective $R$-module morphism $\wt{R} \rarr R$ sending $\eps \otimes 1$ to $1$ and $1 \otimes r$ to $r$. This map induces a surjective $R$-module morphism $q:\rH_{\eps,c} \twoheadrightarrow \rH_{0,c}$. Now, the Poisson bracket of $a,b \in \rZ(\rH_{0,c})$ is defined as $\lbrace a,b \rbrace \dopgleich q(\lbrack \wt{a},\wt{b} \rbrack)$. As the implementation of rational Cherednik algebras in \textsc{Champ} supports general base rings and parameters, we are also able to compute Poisson brackets in \textsc{Champ}.

\section{Restricted rational Cherednik algebras} \label{restricted_cherednik}

\begin{parani}
Besides the capability of performing computations in rational Cherednik algebras it is one aim of \champ\ to compute representation theoretic properties of \textit{restricted} rational Cherednik algebras. These algebras—which were first seriously studied by Gordon \cite{Gor-Baby-verma}—are finite-dimensional quotients of $\rH_{0,c}$ by a centrally generated ideal and they possess (partially established, partially conjectural) relations to Hecke algebras. These relations are one reason for studying (restricted) rational Cherednik algebras. In this section, we will review the basic properties of these algebras, explain what representation theoretic problems we are interested in, and address some computational issues. We include a quick review of Martino's conjecture to be very precise about what we computed and to ensure that these computations yield proofs of this conjecture in the cases under consideration.
\end{parani}

\subsection{Restricted rational Cherednik algebras} \label{rrca}

The $\bbN$-graded ring 
\[
\rZ_\Gamma \dopgleich K \lbrack V \rbrack^G \otimes_K K \lbrack V^* \rbrack^G \subs K \lbrack V \oplus V^* \rbrack^G
\]
of \word{bi-invariants} maps under the PBW morphism into the center of $\rH_{0,c}$ and embeds the scalar extension $\rZ_\Gamma^R$ as a central subalgebra of $\rH_{0,c}$. For $K = R = \bbC$ this was proven by Etingof–Ginzburg \cite{EG-Symplectic-reflection-algebras}, and Gordon's proof \cite{Gor-Baby-verma} in this case also works without modifications in our general setting. We can thus view $\rH_{0,c}$ as a $\rZ_\Gamma^R$-algebra. Note that since $R$ is a flat $K$-module, the scalar extension $\rZ_\Gamma^R$ is simply given by replacing $K$ by $R$ above. As the extension $K \lbrack V \rbrack^G \subs K \lbrack V \rbrack$ is finite (see \cite[12.27]{GorWed10-Algebraic-geomet}), the PBW theorem implies that $\rH_{0,c}$ is a finite $\rZ_\Gamma^R$-module. The finiteness implies (see \cite[\S6, \S17]{Thiel.UOn-restricted-ration}) that we have a decomposition
\begin{align}  \label{simple_decomposition}
\Simp(\rH_{0,c}) = \coprod_{\fm \in \Max(\rZ_\Gamma^R)} \Simp(\rH_{0,c}(\fm))
\end{align}
of the set of isomorphism classes of simple modules, where $\Max$ denotes the maximal ideal spectrum and $\rH_{0,c}(\fm) \dopgleich \rH_{0,c}/\fm\rH_{0,c}$ is the \word{specialization} of $\rH_{0,c}$ in $\fm \in \Max(\rZ_\Gamma^R)$. This decomposition follows essentially from the fact that maximal ideals and left primitive ideals coincide in $\rH_{0,c}$ as it is a PI ring. The advantage is that on the right hand side we have finite-dimensional algebras over fields which might be easier to study than $\rH_{0,c}$ itself. \\

Let
\[
\fa_{\Gamma}^R \dopgleich (\rZ_\Gamma^R)_+ = (R \lbrack V \rbrack^G_+ \otimes_R R \lbrack V ^* \rbrack^G) + (R \lbrack V \rbrack^G \otimes_R R \lbrack V^* \rbrack^G_+) 
\]
be the augmentation ideal of $\rZ_\Gamma^R$. The quotient $\ol{\rH}_c \dopgleich \rH_{0,c}/\fa_\Gamma^R \ol{\rH}_c$ is called the \word{restricted rational Cherednik algebra} of $\Gamma$ in $c$. Note that $\rZ_\Gamma^R/\fa_\Gamma^R \cong R$, so $\fa_\Gamma^R$ is maximal if and only if $R$ is a field. In this case, $\ol{\rH}_c$ is one of the specializations in the decomposition (\ref{simple_decomposition}). 

Recall that the \word{coinvariant algebra} $K \lbrack V \rbrack_G$ of $\Gamma$ is the quotient of $K \lbrack V \rbrack$ by the \word{Hilbert ideal} $\fh_\Gamma$, which is the ideal in $K \lbrack V \rbrack$ generated by the augmentation ideal $K \lbrack V \rbrack^G_+$ of $K \lbrack V \rbrack^G$. It follows at once from the PBW theorem that the PBW morphism induces an $R$-module isomorphism
\begin{align} \label{che_res_triang_dec}
R \lbrack V \rbrack_G \otimes_R RG \otimes_R R \lbrack V^* \rbrack_G \cong \ol{\rH}_c \;,
\end{align}
implying that $\ol{\rH}_c$ is a free $R$-module with
\[
\dim_R \ol{\rH}_c = \dim_R R \lbrack V \rbrack_G \cdot |G| \cdot \dim_R R \lbrack V^* \rbrack_G \;.
\]
In case both $K \lbrack V \rbrack^G$ and $K \lbrack V^* \rbrack^G$ are polynomial (this holds for example in the non-modular setting by a theorem by Bourbaki–Chevalley–Serre as $\Gamma$ is a reflection group), the extensions $K \lbrack V \rbrack^G \subs K \lbrack V \rbrack$ and $K \lbrack V^* \rbrack^G \subs K \lbrack V^* \rbrack$ are free of dimension equal to $|G|$. This implies that in this case $\dim_R \ol{\rH}_c = |G|^3 = \dim_{\rZ_\Gamma^R}\rH_{0,c}$. \\

\subsection{Computing in restricted rational Cherednik algebras} 
 \label{monomial_basis_coinv}
 
Fix a Gröbner basis of the Hilbert ideal of $\Gamma$ with respect to some monomial order. As in Example \ref{example_groebner_basis_normal_form} this allows us to compute a monomial basis $(\ol{\bix}^\lambda)_{\lambda \in \Lambda}$ of the coinvariant algebra $K \lbrack V \rbrack_G$, where $\Lambda \subs \bbN^n$ is some finite subset and $\ol{\bix} \dopgleich (\ol{x}_i)_{i=1}^n$ are the images of the $x_i \in K \lbrack V \rbrack$ in $K \lbrack V \rbrack_G$. Similarly, we obtain a monomial basis $(\ol{\biy}^\sigma)_{\sigma \in \Sigma}$ of $K \lbrack V^* \rbrack_G$. Then by the above $\ol{\rH}_c$ is a free $R$-module with basis $(\ol{\bix}^\lambda \ol{\biy}^\sigma g)_{\lambda \in \Lambda, \sigma \in \Sigma, g \in G}$ and we call a basis of this form a \word{PBW basis} of $\ol{\rH}_c$.  
Algorithm \ref{alg_computing_in_cherednik} can easily be modified to compute PBW basis representations of products in $\ol{\rH}_c$—we just have to work in the group algebra $(R \lbrack V \rbrack_G \otimes_R R \lbrack V^* \rbrack_G) G$. This is again supported by \textsc{Champ}.

\subsection{Representation theory} \label{rrca_represent_theory}

Now, we turn our attention to representation theoretic problems of $\ol{\rH}_c$ which are originally due to Gordon \cite{Gor-Baby-verma}. First of all, note that $R \langle V \oplus V^* \rangle \rtimes RG$ is naturally a $\bbZ$-graded $R$-algebra by putting $V^*$ in degree $1$, $G$ in degree $0$, and $V$ in degree $-1$. The elements in (\ref{cherednik_relation_1}) to (\ref{cherednik_relation_3}) defining the ideal $\rI_{0,c}$ are all homogeneous so that $\rH_{0,c}$ inherits this $\bbZ$-grading. Since the Hilbert ideals are homogeneous, it follows moreover that the restricted rational Cherednik algebra $\ol{\rH}_c$ also inherits this $\bbZ$-grading. 
Gordon \cite{Gor-Baby-verma} observed that the triangular decomposition (\ref{che_res_triang_dec}) of $\ol{\rH}_c$ governs its representation theory by employing a general theory of Holmes–Nakano \cite{HN-Brauer-type-rec}. First note that due to the PBW theorem both the $R$-algebras $\ol{\rH}_{c,m} \dopgleich RG$ and $\ol{\rH}_{c,r} \dopgleich RG \ltimes R \lbrack V^* \rbrack_G$ naturally embed as subalgebras in $\ol{\rH}_c$. This is the ``middle part'' and the ``right Borel subalgebra'' of the triangular decomposition (\ref{che_res_triang_dec}), respectively. Mapping elements of $V$ to zero yields a surjective algebra morphism $q_{c,r}:\ol{\rH}_{c,r} \twoheadrightarrow \ol{\rH}_{c,m}$ and by $q_{c,r*}$ we denote the induced inflation functor $\cat{(gr)mod}{\ol{\rH}_{c,m}} \rarr \cat{(gr)mod}{\ol{\rH}_{c,r}}$. The key tool is now the \word{Verma functor}
\[
\Delta_c \dopgleich \ol{\rH}_c \otimes_{\ol{\rH}_{c,r}} q_{c,r*}(-) : \cat{(gr)mod}{\ol{\rH}_{c,m}} \rarr \cat{(gr)mod}{\ol{\rH}_c} 
\]
between categories of finitely generated (graded) modules. It is not hard to see that 
\begin{equation}
\Delta_c(W) \cong R \lbrack V \rbrack_G \otimes_R W
\end{equation}
as $R$-modules provided that $W$ is free as an $R$-module (see \cite{HN-Brauer-type-rec} or \cite[\S18]{Thiel.UOn-restricted-ration}). 

Now, suppose that $R$ is a field and that $KG$ splits (the latter holds for example if $K$ is of characteristic zero by a theorem by Benard \cite{Ben-Schur-indices-and-splitti-0}). Although Holmes–Nakano \cite{HN-Brauer-type-rec} assume for their theory an algebraically closed base field, their arguments also work when the algebra is just split (see \cite[\S18]{Thiel.UOn-restricted-ration}) and show that for each simple $KG$-module $\lambda$ the corresponding Verma module $\Delta_c(\lambda) \dopgleich \Delta_c(\lambda^R)$ of $\ol{\rH}_c$ is an indecomposable module with simple head $\rL_c(\lambda)$ and that  $(\rL_c(\lambda))_{\lambda \in \Simp(KG)}$ is a system of representatives of the simple $\ol{\rH}_c$-modules. The Verma module $\Delta_c(\lambda)$ is naturally graded and it has been proven in \cite{HN-Brauer-type-rec} that its radical is a graded submodule. Hence, $\rL_c(\lambda)$ is naturally graded, too. Arguments by Bonnafé–Rouquier \cite[Proposition 9.2.5]{Bonnafe.C;Rouquier.R13Cellules-de-Calogero} furthermore show that $\ol{\rH}_c$ itself splits. There is now a natural correspondence between simple $KG$-modules and simple $\ol{\rH}_c$-modules and so the distribution of simple $\ol{\rH}_c$-modules into the blocks of $\ol{\rH}_c$ yields a partition $\mrm{CM}_c$ of the set of simple $KG$-modules whose members are called the \word{Calogero–Moser $c$-families}. 

\subsection{Gordon's questions}

 \label{gordon_problems}
Gordon formulated in \cite[\S7]{Gor-Baby-verma} the following questions concerning the representation theory of $\ol{\rH}_c$ for a parameter $c$ with values in an extension field $R$ of $K$:
\begin{enum_thm}
\item \label{gordon_problems:g_char} What is the graded $G$-character of the simple modules $\rL_c(\lambda)$? This includes knowing their dimensions and their Poincaré series.
\item \label{gordon_problems:verma_dec} What are the composition factors of the Verma modules $\Delta_c(\lambda)$?
\item What are the Calogero–Moser $c$-families? \\
\end{enum_thm}
These questions are so far only studied for $K=R=\bbC$ and we cannot go into details about what is already known in this case (see \cite[16.2, 16.4]{EG-Symplectic-reflection-algebras}, \cite{Finkelberg.M;Ginzburg.V02Calogero-Moser-space}, \cite[6.4, 7.3]{Gor-Baby-verma}, \cite[\S3.3]{Bel-Singular-CM}, \cite{Mar-CM-Rouqier-partition}, \cite{Martino.M11Blocks-of-restricted}, and \cite{Thiel.UOn-restricted-ration}). The point is that almost nothing is known for \textit{exceptional} complex reflection groups and this was one reason for the development of \champ.

\subsection{The generic situation} \label{problem_reductions}
The above problems are formulated for parameters $c$ with values in an extension field $R$ of $K$, i.e., for points of the affine $K$-scheme $\fR_\Gamma \dopgleich \bbA_K^{\# \sC_\Gamma}$. This infinite amount of parameters would be a serious issue for a computational approach but the following two facts allow us to reduce this to finitely many problems. First of all, it is proven in \cite{Thi-Decomposition-morphi2014} that decomposition morphisms are generically trivial. This means essentially that once we know the solution to \S\ref{gordon_problems}\ref{gordon_problems:g_char} and \S\ref{gordon_problems}\ref{gordon_problems:verma_dec} for the generic point $\bic$ of $\fR_\Gamma$—i.e., $\bic$ is the family of indeterminates of the rational function field $K( (c_s)_{s \in \sC_\Gamma})$—then we know the solution for all $c$ in a non-empty open subset of $\fR_\Gamma$. This \word{generic situation} is really the starting point of computational considerations and is supported by \champ. Similarly, it is proven in \cite{Bonnafe.C;Rouquier.R13Cellules-de-Calogero} (see also \cite[\S11]{Thiel.UOn-restricted-ration}) that blocks show the same behavior, meaning that once we know the generic Calogero–Moser families $\mrm{CM}_\bic$, we know them for all $c$ in a non-empty open subset of $\fR_\Gamma$. After the generic situation is understood, we have to determine the locus of ``exceptional parameters'' and continue the above process. This is exactly how we proceed for the groups $\rG_4$,  $\rG_{13}$, and $\rG_{20}$ to compute the answers to Gordon's questions for all parameters.

\subsection{Martino's conjecture} \label{martino_conjecture_subsection}

 \label{martino_prep}
Before we discuss our approach to the computational solution of Gordon's questions, let us first explain why the Calogero–Moser families are interesting. To this end, we need a different type of parameters for rational Cherednik algebras due to Ginzburg–Guay–Opdam–Rouquier \cite{GinGuaOpd-On-the-category-scr-O-for-0}. Let $\sA_\Gamma$ be the set of $G$-orbits of reflection hyperplanes of $\Gamma$. For a reflection hyperplane $H$ of $\Gamma$ the stabilizer subgroup $G_H$ is cyclic of some order $e_H$ prime to the characteristic of $K$. This order is constant along the $G$-orbit $\Omega$ of $H$ so that we can denote it by $e_\Omega$. We denote by $\boldsymbol{\Omega}_\Gamma$ the set of pairs $(\Omega,j)$ with $\Omega \in \sA_\Gamma$ and $1 \leq j \leq e_\Omega-1$, and denote by $\ol{\boldsymbol{\Omega}}_\Gamma$ the set of pairs $(\Omega,j)$ with $0 \leq j \leq e_\Omega-1$. Let $\ol{\fR}_\Gamma$ be the affine $K$-scheme $\bbA_K^{\#\ol{\boldsymbol{\Omega}}_\Gamma}$. For $k\in \ol{\fR}_\Gamma(R)$ we now define a function $c_k:\sC_\Gamma \rarr R$ by
\begin{align} \label{ggor_parameter}
c_k(s) \dopgleich \sum_{j=0}^{e_{\Omega_s-1}} \mrm{det}(s)^j \left( k_{\Omega_s,j+1} - k_{\Omega_s,j} \right) \;,
\end{align}
where $\Omega_s$ is the $G$-orbit of the reflection hyperplane of $s$ and we consider the index $j$ always modulo $e_{\Omega_s}$. We set $\rH_{0,k} \dopgleich \rH_{0,c_k}$. It is not hard to see that (\ref{ggor_parameter}) yields a surjective $R$-linear map $\Phi_\Gamma(R): \ol{\fR}_\Gamma(R) \rarr \fR_\Gamma(R)$, and that this defines a surjective $K$-scheme morphism $\Phi_\Gamma:\ol{\fR}_\Gamma \rarr \fR_\Gamma$. We can thus think of $\ol{\fR}_\Gamma$ as an artificial extension of the parameter space for restricted rational Cherednik algebras of $\Gamma$. On $\ol{\fR}_\Gamma$ we define an involution $(\cdot)^\sharp$ by $k^\sharp \dopgleich (k_{\Omega,-j})$. The closed subscheme $\ol{\fR}_\Gamma^0$ of $\ol{\fR}_\Gamma$ consisting of all $k$ with $k_{\Omega,0} = 0$ is stable under this involution and we call its points \word{Cherednik parameters of GGOR type} for $\Gamma$. Note that $\Phi_\Gamma$ restricts to an isomorphism between $\ol{\fR}_\Gamma^0$ and $\fR_\Gamma$ so that this can be considered as a re-parametrization of $\fR_\Gamma$. \\

Now, assume that $K$ is of characteristic zero and that $\Gamma$ is irreducible. Chlouveraki's \cite{Chl09-Blocks-and-famil} essential hyperplanes define a union $\ol{\sE}_\Gamma$ of hyperplanes in $\ol{\fR}_\Gamma$ defined by integral equations, and attached to any point $k \in \ol{\fR}_\Gamma$ is a partition $\mrm{Rou}_k$ of the simple $KG$-modules whose members are called the \word{Rouquier $k$-families}. We cannot go into details about Rouquier families here (see \cite{BroKim02-Familles-de-cara}, \cite{MalRou-Familles-de-caracteres-de-0}, \cite{Chl09-Blocks-and-famil}, and in particular \cite{Bonnafe.C;Rouquier.R13Cellules-de-Calogero} for the most general discussion) and just note how we can define them for a general base field $K$ of characteristic zero instead of just $K=\bbC$. To this end, we have to choose a realization $\Gamma'$ of $\Gamma$ over the complex numbers, which is possible as $\Gamma$ admits a realization over its character field. When doing this we have to keep track of the orbits of hyperplanes of reflections to avoid changing the parameters. Then Chlouveraki's theory defines the essential hyperplanes in $\ol{\fR}_{\Gamma'}$  and the Rouquier $k$-families for any $k \in \ol{\fR}_{\Gamma'}(\bbC)$. These families are already uniquely determined by the essential hyperplanes $k$ lies on. This and the fact that the essential hyperplanes are defined by integral equations allows us to transport the essential hyperplanes to $\ol{\fR}_\Gamma$  and to define Rouquier families for any point of $\ol{\fR}_\Gamma$. We remark that for the definition of Rouquier families we tacitly assume the validity of some standard assumptions about Hecke algebras (see \cite[4.2.3]{Chl09-Blocks-and-famil}) which are not known to hold for all exceptional complex reflection groups. The interest in Calogero–Moser families is now justified by the following conjecture.

\begin{conjecture}[Martino, \cite{Mar-CM-Rouqier-partition}] \label{martino_conjecture}
Assume that $K$ is of characteristic zero and that $\Gamma$ is irreducible. The following holds:
\begin{enum_thm}
\item $\mrm{Rou}_{k^\sharp}$ is a refinement of $\mrm{CM}_k \dopgleich \mrm{CM}_{c_k}$ for any $k \in \ol{\fR}_\Gamma$.
\item There is a non-empty open subset $U$ of $\ol{\fR}_\Gamma$ such that $\mrm{Rou}_{k^\sharp} = \mrm{CM}_k$ for all $k \in U$.
\end{enum_thm}
\end{conjecture}

We call the first part of the conjecture the \word{special parameter conjecture} and the second part the \word{generic parameter conjecture}. Because restricted rational Cherednik algebras and cyclotomic Hecke algebras always split, it is enough to consider some particular realization of each type of complex reflection groups in the Shephard–Todd classification and then prove the conjecture for $K$-points. 
Furthermore, we note that for $k \in \ol{\fR}_\Gamma$ the $k$-cyclotomic Hecke algebra is naturally isomorphic to the $k^0$-cyclotomic Hecke algebra, where $k^0$ is obtained from $k$ by setting $k_{\Omega,0}$ to zero for all $\Omega$ (this follows from \cite[2.1.13]{Bonnafe.C;Rouquier.R13Cellules-de-Calogero}). Hence, we can equivalently consider the conjecture just for points of $\ol{\fR}_\Gamma^0$ as originally formulated by Martino. 
Due to the behavior of Calogero–Moser families explained in \S\ref{problem_reductions} and due to the behavior of Rouquier families explained in \S\ref{martino_prep} the generic parameter conjecture is equivalent to $\mrm{Rou}_{\bik} = \mrm{CM}_\bik$, where $\bik$ is the generic point of $\ol{\fR}_\Gamma^0$.  \\

Martino's conjecture is known to be true for symmetric and imprimitive complex reflection groups by \cite{Mar-CM-Rouqier-partition}, \cite{Bel-Singular-CM}, and \cite{Martino.M11Blocks-of-restricted}. The generic parameter conjecture is known to be true for $\rG_4$ by \cite{Bel-Singular-CM}, and also for $\rG_5$, $\rG_6$, $\rG_8$, $\rG_{10}$, $\rG_{23}$, $\rG_{24}$, and $\rG_{26}$ by \cite{Thiel.U13A-Counter-Example-to}. It was shown in \cite{Thiel.U13A-Counter-Example-to}, however, that the generic parameter conjecture fails for $\rG_{25}$. In all cases where this conjecture is known to hold it was proven by determining the Calogero–Moser families and comparing them to the Rouquier families which have been determined by Chlouveraki \cite{Chl09-Blocks-and-famil}. So far, there is no theoretical explanation for this connection, and the counter-example in case $\rG_{25}$ makes it even harder to understand the situation. 

\subsection{Euler families}  \label{euler_families}

Bonnafé–Rouquier \cite{Bonnafe.C;Rouquier.R13Cellules-de-Calogero} have  pointed out a neat argument why there could exist a connection between Calogero–Moser families and Rouquier families at all.  First of all, the \word{Euler element} of $\rH_{0,c}$, introduced in  \cite{GinGuaOpd-On-the-category-scr-O-for-0}, is defined as
\begin{equation}
\mrm{eu}_c = \sum_{i=1}^n y_ix_i + \sum_{s \in \mrm{Ref}_\Gamma} \frac{1}{\eps_s-1} c(s) s = \sum_{i=1}^n x_iy_i + \frac{\eps_s}{\eps_s-1}c(s) s \;,
\end{equation}
where as usual $(y_i)_{i=1}^n$ is a basis of $V$ with dual basis $(x_i)_{i=1}^n$, and where $\eps_s$ denotes the non-trivial eigenvalue of $s$. The definition does not depend on the choice of a basis. This element is known to be central and its image in $\ol{\rH}_c$ is again a non-trivial central element. Let $\Omega_\lambda^c$ be the central character of the simple $\ol{\rH}_c$-module $\rL_c(\lambda)$. Then the values of these characters on the Euler element yield a partition $\mrm{Eu}_c$ of the simple $KG$-modules which is coarser than $\mrm{CM}_c$. We call its members the \word{Euler $c$-families}. It is proven in \cite[10.2.2]{Bonnafe.C;Rouquier.R13Cellules-de-Calogero} that for $k \in \ol{\fR}_\Gamma$ the equality $\Omega_\lambda^k(\mrm{eu}_k) = c_\lambda(k^\sharp)$ holds, where $c_\lambda(k^\sharp)$ is a constant multiple of the ``$q$-logarithm'' of the value of the central character of the simple module belonging to $\lambda$ of the $k^\sharp$-cyclotomic Hecke algebra on the central element $\pi$ coming from the center of the braid group of $\Gamma$ (see \cite[\S2.2]{Bonnafe.C;Rouquier.R13Cellules-de-Calogero}). These values define similarly a partition $\Pi_{k^\sharp}$ of the simple $KG$-modules which is coarser than $\mrm{Rou}_{k^\sharp}$ and equal to $\mrm{Eu}_k$. We thus have
\begin{equation}
\mrm{CM}_k \leq \mrm{Eu}_k = \Pi_{k^\sharp} \geq \mrm{Rou}_{k^\sharp} \;,
\end{equation}
where $\leq$ denotes refinement. Of course, this does not explain why $\mrm{CM}_k \geq \mrm{Rou}_{k^\sharp}$ should hold.

\subsection{Verma families}  \label{verma_families}
Next to the Euler families there is another type of families giving a further approximation of the Calogero–Moser families. Namely, for a fixed simple $KG$-module $\lambda$ we collect all constituents of $\Delta_c(\lambda)$. For each of these constituents $S_\mu$ we again collect all constituents of $\Delta_c(\mu)$ etc. This process stabilizes and gives us a partition $\mrm{Ver}_c$ of the simple $KG$-modules whose members we call the \word{Verma $c$-families}. As Verma modules are indecomposable, these are always contained in a family coming from a block of $\ol{\rH}_c$, i.e., each Verma family is contained in a Calogero–Moser family, so $\mrm{Ver}_c \leq \mrm{CM}_c$. We thus get a tower
\begin{equation}
\mrm{Ver}_c \leq \mrm{CM}_c \leq \mrm{Eu}_c 
\end{equation}
giving us approximations of $\mrm{CM}_c$ from two sides. The Euler families are easily computable using the characters of the simple $KG$-modules (see \cite{Bonnafe.C;Rouquier.R13Cellules-de-Calogero} or \cite{Thiel.U13A-Counter-Example-to}). The Verma families in turn can be computed in many cases by the methods we discuss in the next paragraphs. Usually, the above tower collapsed, i.e., the Verma families were equal to the Euler families and thus equal to the Calogero–Moser families. %

\begin{remark}
Recently, Bonnafé–Rouquier \cite[\S13.4]{Bonnafe.C;Rouquier.R13Cellules-de-Calogero} haven proven that in case $K$ is of characteristic zero, the Verma families are in fact equal to the Calogero–Moser families (we observed this property before in our explicit computations). This is now the theoretical foundation showing that the key to determining the Calogero–Moser families are the Verma families.
\end{remark}

\section{Computations with Verma modules}

After we discussed the main problems we are interested in—namely Gordon's questions—let us go back to computational issues. Clearly, we first have to find an explicit description of the Verma modules for any computational approach to these problems. We discuss this here along with some aspects about efficient computation of Verma modules. The main problem to solve Gordon's questions is then to be able to compute decomposition matrices of Verma modules. We discuss an abstract strategy in \S\ref{abstract_idea} which we will turn into a serious method in the following three paragraphs.

\subsection{Computing Verma modules}

Let $\rho:G \rarr \End_K(W)$ be a finite-dimensional $K$-representation. Then the Verma module $\Delta_c(\rho)$ is uniquely determined by the action of the generators $\bix \cup \biy \cup \mbi{g} $ of $\ol{\rH}_c$, where $\mbi{g}$ is a generating system of $G$, and as $\Delta_c(\rho)$ is free and finitely generated as an $R$-module, these actions are described by some matrices. In this way a Verma module can be represented in the computer once we have chosen bases and understood the action. To this end, we choose besides a basis $\biy \dopgleich (y_i)_{i=1}^n$ of $V$ with dual basis $\bix \dopgleich (x_i)_{i=1}^n$ and a generating system $\mbi{g} \dopgleich (g_i)_{i=1}^r$ of $G$ also a monomial basis $\ol{\bix}^\Lambda \dopgleich (\ol{\bix}^\lambda)_{\lambda \in \Lambda}$ of $K \lbrack V \rbrack_G$ as described in \S\ref{monomial_basis_coinv}. Furthermore, we fix a basis $\biw \dopgleich (w_k)_{k=1}^d$ of $W$. Then an $R$-basis of $\Delta_c(\rho) \cong R \lbrack V \rbrack_G \otimes_R W$ is formed by the elements $\ol{\bix}^\lambda \otimes w_k$, and with respect to this basis we now describe the action of the generators. \\

First, let us consider the action of $x_i$ on $\Delta_c(\rho)$. We have 
\[
x_i.(\ol{\bix}^\mu \otimes w_k) = (x_i \ol{\bix}^\mu) \otimes w_k \;.    
\]
Hence, if the basis representation of $x_i \ol{\bix}^\mu \in K \lbrack V \rbrack_G$ in the basis $\ol{\bix}^\Lambda$ is
\[
x_i\ol{\bix}^\mu = \sum_{\lambda \in \Lambda} \alpha_{\lambda}^{i,\mu} \ol{\bix}^\lambda \;,
\]
then
\begin{equation} \label{verma_modules_explicitly:x}
x_i.(\ol{\bix}^\mu \otimes w_k) = \sum_{\lambda \in \Lambda} \alpha_{\lambda}^{i,\mu} \ol{\bix}^\lambda \otimes w_k
\end{equation}
is the basis representation of $x_i.(\ol{\bix}^\mu \otimes w_k) \in \Delta_c(\rho)$ in the basis $\ol{\bix}^\Lambda \otimes \biw$. So, we actually just need to understand the action of the $x_i$ on the coinvariant algebra $K \lbrack V \rbrack_G$ and this can computationally be solved using Gröbner bases. \\

Now, let us consider the action of $g_i$ on $\Delta_c(\rho)$. We have 
\[
g_i.(\ol{\bix}^\mu \otimes w_k) = (g_i \ol{\bix}^\mu) \otimes w_k = (\,^{g_i}\ol{\bix}^\mu g_i) \otimes w_k = \,^{g_i} \ol{\bix}^\mu \otimes \,^{g_i}w_k \;.
\]
Hence, if the basis representation of $\,^{g_i}\ol{\bix}^\mu \in K \lbrack V \rbrack_G$ in the basis $\ol{\bix}^\Lambda$ is
\[
\,^{g_i}\ol{\bix}^\mu = \sum_{\lambda \in \Lambda} \beta_{\lambda}^{i,\mu} \ol{\bix}^\lambda
\]
and the basis representation of $\,^{g_i}w_k$ in the basis $\biw$ is
\[
\,^{g_i}w_k = \sum_{t=1}^d \gamma_{t}^{i,k} w_t \;, 
\]
then the basis representation of $g_i.(\ol{\bix}^\mu \otimes w_k)$ in the basis $\ol{\bix}^\Lambda \otimes \biw$ is
\begin{equation} \label{verma_modules_explicitly:g}
    g_i.(\ol{\bix}^\mu \otimes w_k) = \left(\sum_{\lambda \in \Lambda} \beta_{\lambda}^{i,\mu} \ol{\bix}^\lambda \right) \otimes \left(\sum_{t=1}^d \gamma_{t}^{i,k} w_t \right) = \sum_{\lambda \in \Lambda} \sum_{t=1}^d \beta_{\lambda}^{i,k} \gamma_{t}^{i,k} \ol{\bix}^\lambda \otimes w_t \;.
\end{equation}
So, to understand the action of $g_i$ in $\Delta_c(\rho)$ we need to understand the action of $g_i$ on the coinvariant algebra $K \lbrack V \rbrack_G$ and on the $KG$-module $W$. The first can again be computationally achieved using Gröbner bases, the second is no problem when we have an explicit realization of $\rho$. \\

Now, we come to the hardest part, namely the action of $y_i$ on $\Delta_c(\rho)$. This is the point where the structure of the restricted rational Cherednik algebra enters the game. Namely, to write the element $y_i(\ol{\bix}^\mu \otimes w_k) = (y_i \ol{\bix}^\mu) \otimes w_k$ in the basis $\ol{\bix}^\Lambda \otimes \biw$, we first have to rewrite $y_i\ol{\bix}^\mu$ in the PBW-basis of $\ol{\rH}_c$. Recall from Lemma \ref{commutator_formula} that
\begin{equation}
 \lbrack y_i, \bix^\mu \rbrack = \sum_{t=1}^n \sum_{s \in \mrm{Ref}_\Gamma}\sum_{l=0}^{\mu_t-1} c(s) (y_i,x_t)_s x_1^{\mu_1} \cdots x_{i-1}^{\mu_{t-1}} x_t^l (^s x_t)^{\mu_t-l-1} (^s x_{t+1})^{\mu_{t+1}} \cdots (^s x_n)^{\mu_n} s \;.
\end{equation} 
 Using this formula we get
\begin{align*}
 & y_i(\ol{\bix}^\mu \otimes w_k) = (y_i \ol{\bix}^\mu) \otimes w_k = (\ol{\bix}^\mu y_i + \lbrack y_i, \ol{\bix}^\mu \rbrack) \otimes w_k = (\ol{\bix}^\mu y_i) \otimes w_k + \lbrack y_i, \ol{\bix}^\mu \rbrack \otimes w_k \allowdisplaybreaks[1]\ \\
 &=  \sum_{t=1}^n \sum_{s \in \mrm{Ref}_\Gamma}\sum_{l=0}^{\mu_t-1} c(s) (y_i,x_t)_s x_1^{\mu_1} \cdots x_{t-1}^{\mu_{t-1}} x_t^l (^s x_t)^{\mu_t-l-1} (^s x_{t+1})^{\mu_{t+1}} \cdots (^s x_n)^{\mu_n} s \otimes w_k \allowdisplaybreaks[1]\ \\
 &= \sum_{t=1}^n \sum_{s \in \mrm{Ref}_\Gamma}\sum_{l=0}^{\mu_t-1} c(s) (y_i,x_t)_s x_1^{\mu_1} \cdots x_{t-1}^{\mu_{t-1}} x_t^l (^s x_t)^{\mu_t-l-1} (^s x_{t+1})^{\mu_{t+1}} \cdots (^s x_n)^{\mu_n} \otimes \,^s w_k \allowdisplaybreaks[1]\ \\
 &= \sum_{s \in \mrm{Ref}_\Gamma} \sum_{t=1}^n \sum_{l=0}^{\mu_t-1} c(s) (y_i,x_t)_s x_1^{\mu_1} \cdots x_{t-1}^{\mu_{t-1}} x_t^l (^s x_t)^{\mu_t-l-1} (^s x_{t+1})^{\mu_{t+1}} \cdots (^s x_n)^{\mu_n} \otimes \,^s w_k \;, \numberthis \label{verma_modules_explicitly:y}
\end{align*}
where we used that $(\ol{\bix}^\mu y_i) \otimes w_k = 0$ by definition of $\Delta_c(\rho)$. 
This expression is not yet a basis expression in the basis $\ol{\bix}^\Lambda \otimes \biw$ but by rewriting the elements on the left hand side of the tensor products in the basis $\ol{\bix}^\Lambda$ as above using Gröbner bases and rewriting the elements on the right hand side in the basis $\biw$ immediately gives a basis expression.  Hence, with the formulas in (\ref{verma_modules_explicitly:x}), (\ref{verma_modules_explicitly:g}), and (\ref{verma_modules_explicitly:y}) we can explicitly compute the Verma module $\Delta_c(\rho)$ and represent it in this way in a computer. Note, however, that it still needs an explicit method---like Gröbner bases---to rewrite elements in the coinvariant algebra in terms of a chosen (monomial) basis.

\subsection{X-tables} \label{X_tables}
Some parts of formula (\ref{verma_modules_explicitly:y}) occur multiple times. In particular if one wants to consecutively compute Verma modules for different $KG$-representations one can split off these parts to increase efficiency. We propose the following approach. Fix $i \in \lbrack 1,n \rbrack$, $s \in \mrm{Ref}_\Gamma$, and $\mu \in \Lambda$. Let $X^{(i,s)}_\mu = (X^{(i,s)}_{\mu,\eta})_{\eta \in \Lambda}$ be such that  $X^{(i,s)}_{\mu,\eta}$ is the coefficient of $\ol{\bix}^\eta$ in the basis representation of  
    \[
    \sum_{t=1}^n \sum_{l=0}^{\mu_{t}-1} (y_i,x_t)_s x_1^{\mu_{1}} \cdots x_{t-1}^{\mu_{t-1}} x_t^l (^s x_t)^{\mu_{t}-l-1} (^s x_{t+1})^{\mu_{t+1}} \cdots (^s x_n)^{\mu_{n}} \in K \lbrack V \rbrack_G
    \]
    in the basis $\ol{\bix}^\Lambda$. We can consider $X^{(i,s)}_\mu$ as a row vector and by varying $\mu$ we get a matrix $X^{(i,s)} \in \Mat_{\Lambda \times \Lambda}(K)$ satisfying
    \begin{equation}
    y_i(\ol{\bix}^\mu \otimes w_k) = \sum_{s \in \mrm{Ref}_\Gamma} c(s) \sum_{\eta \in \Lambda} X^{(i,s)}_{\mu,\eta} \ol{\bix}^\eta \otimes \,^s w_k \;.
    \end{equation}
   Note that the matrix $X^{(i,s)}$ is independent of the representation $\rho$ and even of $c$ so that it can be used again for further computations.
For the computation of $X^{(i,s)}$ we can define for fixed $\mu \in \Lambda$ the following two expressions, indexed by $t \in \lbrack 1,n \rbrack$: 
   \begin{equation}
   p_\mu^{\tn{start}}(t) \dopgleich x_1^{\mu_{1}} \cdots x_{t-1}^{\mu_{t-1}}  \;,
  \end{equation}
  \begin{equation}
      p_\mu^{\tn{end}}(t) \dopgleich x_{t+1}^{\mu_{t+1}} \cdots x_n^{\mu_{n}} \;.
  \end{equation}
   Then for $s \in \mrm{Ref}_\Gamma$ the row vector $X_\mu^{(i,s)}$ can be determined by computing the basis representation of the element
   \begin{equation}
    \sum_{s \in \mrm{Ref}_\Gamma} \sum_{t =1}^n(y_i,x_t)_s p_\mu^{\mrm{start}}(t) \left(\sum_{l=0}^{\mu_{t}-1} x_t^l (^s x_t)^{\mu_{t}-l-1} \right) \,^s p_\mu^{\mrm{end}}(t) \;.
   \end{equation}

The above methods for computing $\Delta_c(\rho)$ are implemented in exactly this way in \champ. To use the grading of Verma modules we implemented a new type \code{ModGr} allowing to handle graded modules in general. Moreover, we observed that Verma modules are usual very sparse and so we use sparse matrices in our implementation. Even Verma modules of dimension a few thousand can in this way be computed quite fast and with low memory usage.

\subsection{Decomposing Verma modules—the abstract idea} \label{abstract_idea}

After this initial problem being solved, we turn to the actual questions in \S\ref{gordon_problems}, namely: how can we compute the simple modules $\rL_c(\lambda)$, i.e., the heads of the Verma modules, and how can we compute the constituents of the Verma modules? Over finite fields, this can be achieved using the \textsc{MeatAxe} algorithm (see \cite{Par-Meataxe}, \cite{HR-Meataxe}, \cite{Hol-Meataxe},  \cite[\S7.4]{Holt.D;Eick.B;OBrien.E05Handbook-of-computat}, \cite[\S1.3]{LP-representations}, \cite[7.1.1]{Geck.M;Jacon.N11Representations-of-H}), which is also implemented in \textsc{Magma}. In the generic situations (where the base ring is a rational function field) and in case of base rings of characteristic zero, however, there does not exist any practical algorithm capable of solving our problems. Although there are some recent approaches to a ``characteristic zero \textsc{MeatAxe}''—so for example the general method developed by Steel \cite{Steel.A12Construction-of-Ordi}, which is also implemented in \textsc{Magma}—no existing algorithm was successful even in smaller examples (see the experiments in \S\ref{magmas_algo_timing} proving this). We therefore conceived a method aiming to solve this problem. Although our whole idea is based on necessary conditions so that the resulting algorithm might not produce a result at all, it turned out to be extremely successful and efficient for Verma modules of restricted rational Cherednik algebras and was the key tool of our progress on Gordon's questions for exceptional complex reflection groups (see \S\ref{further_results}). \\

Our approach is very general—so, it has nothing to do with Cherednik algebras—and relies on the fact that we can solve the problems over finite fields using the \textsc{MeatAxe}. As we do not have a finite field at hand, we first need a way to transfer the situation to a finite field and then we have to figure out what the situation over the finite field tells us about our original situation. The following proposition---formulated abstractly---is the main ingredient for our approach.

\begin{definition}
If $\sA$ and $\sB$ are two essentially small abelian categories, then a group morphism $d:\rK_0(\sA) \rarr \rK_0(\sB)$ of the zeroth $\rK$-groups is called \word{positive} if $d(\rK_0^+(\sA)) \subs \rK_0^+(\sB)$, where $\rK_0^+$ is the submonoid represented by objects, and it is called \word{strongly positive} if it is positive and $d(\lbrack X \rbrack) = 0$ implies $\lbrack X \rbrack = 0$ for all $\lbrack X \rbrack \in \rK_0^+(\sA)$.
\end{definition}

\begin{proposition} \label{strongly_positive_map_props}
 Let $\sA$ and $\sB$ be two abelian categories of finite length and let $d:\rK_0(\sA) \rarr \rK_0(\sB)$ be a strongly positive morphism. Let $X \in \sA$. The following holds:
\begin{enum_thm}
 \item If $d(\lbrack X \rbrack)$ is simple, then $X$ itself is simple.
 \item \label{strongly_positive_map_props_1} Let $(S_i)_{i \in I}$ be a set of representatives of the simple objects of $\sA$ and let $(T_j)_{j \in J}$ be a set of representatives of the simple $\sB$-objects. Let $X \in \sA$ and let $J_X^d \dopgleich \lbrace j \in J \mid \lbrack d(\lbrack X \rbrack) : T_j \rbrack \neq 0 \rbrace$. Suppose that there exists a subset $I_X^d \subs I$ such that $\lbrack X:S_i \rbrack = 0$ for all $i \in I \setminus I_X^d$ and such that there exists a bijection $\lambda_X^d:J_X^d \rarr I_X^d$ with $d(\lbrack S_{\lambda_X^d(j)} \rbrack) = T_j$ for all $j \in J_X^d$. Then
\[
\lbrack X \rbrack = \sum_{j \in J_X^d} \lbrack d(\lbrack X \rbrack): T_j \rbrack \lbrack S_{\lambda_X^d(j)} \rbrack
\]
and in this case we say that $d$ is \word{$X$-generic}.
\end{enum_thm}
\end{proposition}

\begin{proof} \hfill

\begin{enum_proof}
 \item Suppose that $X$ is not simple. Then we can write $\lbrack X \rbrack = \lbrack X_1 \rbrack + \lbrack X_2 \rbrack$ with $\lbrack X_1 \rbrack, \lbrack X_2 \rbrack \neq 0$ and we get the relation $\lbrack T \rbrack = d(\lbrack X \rbrack) = d(\lbrack X_1 \rbrack) + d(\lbrack X_2 \rbrack)$ in $\rK_0^+(\sB)$ with $T \in \sB$ simple. Since $d$ is strongly positive, we have $d(\lbrack X_1 \rbrack),d(\lbrack X_2 \rbrack) \neq 0$. But then the above relation in $\rK_0^+(\sB)$ is impossible. Hence, $X$ must be simple.

\item The basis representation of $\lbrack X \rbrack$ is
\[
\lbrack X \rbrack = \sum_{i \in I} \lbrack X:S_i \rbrack \lbrack S_i \rbrack = \sum_{i \in I_X^d} \lbrack X:S_i \rbrack \lbrack S_i \rbrack.
\]
Using the fact that $\lambda_X^d$ is a bijection, we get
\begin{align*}
d(\lbrack X \rbrack)  &= \sum_{i \in I_X^d} \lbrack X:S_i \rbrack d(\lbrack S_i \rbrack) = \sum_{j \in J_X^d} \lbrack X:S_{\lambda_X^d(j)} \rbrack d(\lbrack S_{\lambda_X^d(j)} \rbrack) = \sum_{j \in J_X^d} \lbrack X:S_{\lambda_X^d(j)} \rbrack \lbrack T_j \rbrack.
\end{align*}
Since the basis representation of $d(\lbrack X \rbrack)$ is
\[
d(\lbrack X \rbrack) = \sum_{j \in J} \lbrack d(\lbrack X \rbrack):T_j \rbrack \lbrack T_j \rbrack= \sum_{j \in J_X^d} \lbrack d(\lbrack X \rbrack):T_j \rbrack \lbrack T_j \rbrack,
\]
the claim is proven. \qedhere

\end{enum_proof}
\end{proof}

\begin{para}
For a finite-dimensional algebra $A$ over a field we denote by $\rG_0(A) \dopgleich \rK_0(A\tn{-}\msf{mod})$ the Grothendieck group of $A$, where $A\tn{-}\msf{mod}$ is the category of finite-dimensional left $A$-modules. Applying Proposition \ref{strongly_positive_map_props} to the Grothendieck groups of finite-dimensional algebras $A$ and $B$ over fields shows us that if we have a strongly positive morphism $d:\rG_0(A) \rarr \rG_0(B)$ and we can compute decompositions in $\rG_0(B)$---for example if the base field of $B$ is finite using the \textsc{MeatAxe}!---, then we can computationally prove that an $A$-module is simple and we have a chance of computing decompositions of $A$-modules in $\rG_0(A)$. The morphism $d$ is really the link between a computationally manageable ring $B$ and the ring $A$. Our proposition leads us to the following two strategies.
\end{para}

\begin{strategy} \label{computing_heads_general}
For computing the head of a finite-dimensional module $V$ with simple head over a finite-dimensional algebra $A$ over a field we propose the following method:
\begin{enum_thm}
    \item Find a strongly positive morphism $d:\rG_0(A) \rarr \rG_0(B)$ with $B$ a finite-dimensional algebra over a \textit{finite} field.
    \item \label{computing_heads_general:radical} Create a submodule $J$ of $V$, which is to be considered as a candidate for the radical, compute the quotient $V/J$ and check using the \textsc{MeatAxe} if $d(V/J)$ is irreducible. If it is, then we know that $V/J$ is simple and is therefore the head of $V$.
\end{enum_thm}   
\end{strategy}

\begin{strategy} \label{computing_const_abstract}
 Let $A$ be a finite-dimensional algebra over a field. Suppose that we have a family $(V_\lambda)_{\lambda \in \Lambda}$ of finite-dimensional $A$-modules with simple heads $(S_\lambda)_{\lambda \in \Lambda}$. Suppose furthermore that this family is \word{constituent-closed}, meaning that every constituent of a member $V_\lambda$ of this family is the head $S_\mu$ of some $V_\mu$. We then have
 \[
 \lbrack V_\lambda \rbrack = \sum_{\mu \in \Lambda} m_{\lambda,\mu} \lbrack S_\mu \rbrack \in \rG_0(A)
 \]
 for some $m_{\lambda,\mu} \in \bbN$ and we propose the following method for computing these decomposition numbers:
\begin{enum_thm}
    \item Find a strongly positive morphism $d:\rG_0(A) \rarr \rG_0(B)$ with $B$ a finite-dimensional algebra over a \textit{finite} field such that $d(S_\lambda)$ is simple for all $\lambda \in \Lambda$.
    \item For each $\lambda \in \Lambda$ compute using the \textsc{MeatAxe} the constituents $(T_{\lambda,\theta})_{\theta \in \Theta_\lambda}$ and their multiplicities $m_{\lambda,\theta}$. Now, check if there exists an injection $\iota_\lambda:\Theta_\lambda \hookrightarrow \Lambda$ such $d(S_\mu) \cong T_{\lambda,\theta}$ for some $\mu \in \Lambda$ and $\theta \in \Theta_\lambda$ if and only if $\mu = \iota_\lambda(\theta)$. In this case
    \[
    \lbrack V_\lambda \rbrack = \sum_{\mu \in \Lambda} m_{\lambda,\mu} \lbrack S_\mu \rbrack \in \rG_0(A) \;, 
    \]
    where $m_{\lambda,\iota_\lambda(\theta)} \dopgleich m_{\lambda,\theta}$ for $\theta \in \Theta_\lambda$ and $m_{\lambda,\mu} \dopgleich 0$ for all $\mu \notin \Im \iota_\lambda$.
\end{enum_thm}
\end{strategy}

\begin{para}
While decomposition morphisms---more precisely, compositions of decomposition morphisms which do not necessarily have to be decomposition morphisms themselves, whence the formulation using strongly positive morphisms---will certainly play a central role for finding appropriate strongly positive morphisms to algebras over finite fields, it is completely unclear at this stage what we should do in Strategy \ref{computing_heads_general}\ref{computing_heads_general:radical} to produce a candidate for the radical of a module with simple head. In the following two paragraphs we will discuss methods to solve these two problems. Our final algorithm is presented in \S\ref{heads}.
\end{para}

\section{Finite field specializations of restricted rational Cherednik algebras} \label{ffspecs}

In this paragraph we discuss a quite general method to produce for an algebra $A$ (satisfying some assumptions) a strongly positive morphism $d:\rG_0(A) \rarr \rG_0(B)$ into the Grothendieck group of a finite-dimensional algebra over a finite field—this is the first step in the strategies outlined in \S\ref{abstract_idea}. In \S\ref{rrca_integral_structures} we discuss when this works for restricted rational Cherednik algebras, and this leads us to the notion of \textit{integral structures} of these algebras.

\subsection{Finite field specializations in general}

\begin{parani}
Let us fix a Dedekind domain $\sO$ with quotient field $K$, a normal commutative $K$-algebra $R$, and an $R$-algebra $A$ which is free and finitely generated as an $R$-module. 
\end{parani}

\begin{definition} \label{finite_field_spec_def}
 A \word{finite field specialization} of $A$ is a pair $(\fm,u)$ such that:
\begin{enum_thm}
    \item $\fm$ is a maximal ideal of $\sO$ with finite residue field.
    \item $u$ is a $K$-point of the $K$-scheme $\Spec(R)$ such that the $K$-algebra $A(u) \dopgleich u^*\!A$ splits and has an $\sO_\fm$-free $\sO_\fm$-structure $\wt{A}(u)$, i.e., the scalar extension $\wt{A}(u)^K$ of $\wt{A}(u)$ to $K$ is isomorphic to $A(u)$.
 \end{enum_thm}  
\end{definition}

Since $A(u)$ splits, the theory of decomposition morphisms by Geck–Rouquier \cite{GR-Centers-Simple-Hecke} and Geck–Pfeiffer \cite[\S7]{GP-Coxeter-Hecke} implies that the decomposition morphism
\[
\rd_{A}^u:\rG_0(A(\biu)) \rarr \rG_0(A(u)) 
\] 
exists, where $\biu$ is the generic point of $\Spec(R)$, i.e., $A(\bu) = A^{\rQ(R)}$, where $\rQ(R)$ is the quotient field of $R$. Now, by assumption $A(u)$ has an $\sO_\fm$-free $\sO_\fm$-structure $\wt{A}(u)$. Since $\sO_\fm$ is a valuation ring, the decomposition morphism
\[
\rd_{\wt{A}(u)}^{\fm_\fm}: \rG_0(A(u)) \rarr \rG_0(\wt{A}(u)(\fm_\fm)) 
\] 
exists, where $\wt{A}(u)(\fm_\fm)$ is the scalar extension of $\wt{A}(u)$ to the residue field of $\fm_\fm$.  As decomposition morphisms are strongly positive, we obtain a strongly positive morphism
\begin{equation} \label{d_of_finite_field_spec}
\begin{tikzcd}
\rG_0(A(\biu)) \arrow{r}{\rd_{A}^u} \arrow[swap,bend right]{rr}{\rd_{A}^{\fm,u}} & \rG_0( A(u) ) \arrow{r}{\rd_{\wt{A}(u)}^\fm} & \rG_0( \wt{A}(u)(\fm_\fm))
\end{tikzcd}
\end{equation} 
We have omitted the choice of the $\sO_\fm$-free $\sO_\fm$-structure of $A(u)$ in the notation $\rd_A^{\fm,u}$ as this will not be important---although the knowledge about the existence of such a structure is of course crucial. We call $\rd_A^{\fm,u}$ the \word{decomposition morphism} of $A$ in $(\fm,u)$ but note that this does not have to be a decomposition morphism itself.

\begin{remark}
The notion of finite field specializations can of course be generalized to arbitrary finite chains of decomposition morphisms ending in the Grothendieck group of an algebra over a finite field. One only has to make sure in each step that the decomposition morphism exists with the main problem being the existence of integral structures.
\end{remark}

\begin{remark} \label{high_chance}
In \cite{Thi-Decomposition-morphi2014} it is proven that decomposition morphisms are generically trivial for finite free algebras with split generic fiber over noetherian normal rings. Hence, assuming that $R$ is noetherian and that $A$ has split fibers, the morphism $\rd_A^{\fm,u}$ is trivial for generic $u$ and for generic $\fm$, meaning that it induces a bijection between simple modules. Hence,  finite field specializations can be used to employ Proposition  \ref{strongly_positive_map_props}\ref{strongly_positive_map_props_1} generically. This already indicates that it makes sense to choose finite field specializations randomly as the probability is quite high to stay in the generic region.
\end{remark}

\begin{remark} \label{mod_finite_field_spec}
If $(\fm,u)$ is a finite field specialization of $A$ and $V$ is a finite-dimensional $A(\biu)$-module, it will be important to explicitly compute a representative of $\rd_A^{\fm,u}(\lbrack V \rbrack)$. To this end, suppose that the image of $u$ is contained in $\sO_\fm$ and that we have an $R_\fp$-free $A_\fp$-structure $\wt{V}$ of $V$ for some $\fp \in \Spec(R)$. Let $\wt{\sV}$ be an $R_\fp$-basis of $\wt{V}$ and let $\sA$ be an $R$-algebra generating system of $A$. If we apply the map $u$ to the entries of the matrices describing the action of $a \in \sA$ on $V$ in terms of the basis $\sV$, we obtain a representative of $\rd_A^u(\lbrack V \rbrack)$. As the image of $u$ is contained in $\sO_\fm$ by assumption, the entries of the matrices just obtained are contained in $\sO_\fm$ and so we can reduce them modulo $\fm_\fm$, and this a representative of $\rd_{\wt{A}(u)}^\fm \circ \rd_A^u(\lbrack V \rbrack) = \rd_A^{\fm,u}(\lbrack V \rbrack)$. In this situation we do not even see the chosen $\sO$-free $\sO$-structure $\wt{A}(u)$ of $A(u)$.

Although formally a bit complicated, this whole process is actually quite straightforward in explicit situations and is automatically performed by the command \code{Specialize} in \champ. That a pair $(\fm,u)$ is indeed a finite field specialization has to be checked manually, however.
\end{remark}

\subsection{Integral structures of restricted rational Cherednik algebras} \label{rrca_integral_structures}

Let us now turn to the problem of finding finite field specializations of restricted rational Cherednik algebras.

\begin{leftbar}
\begin{assumption}
By $\Gamma \dopgleich (G,V)$ we denote a finite reflection group over a field $K$ containing a Dedekind domain $\sO$ with quotient field $K$. We assume as usual that all reflections are diagonalizable. Furthermore, we assume that the action of $G$ on $V$ and on $V^*$ has no non-zero fixed points, i.e., the $G$-modules $V$ and $V^*$ are \textit{essential}. This certainly holds if $\Gamma$ is irreducible. 
\end{assumption}
\end{leftbar}

\begin{definition}
We say that a Cherednik parameter $c \in \fR_\Gamma(K)$ is \word{$\sO$-integral} if the $K$-algebra $\ol{\rH}_c$ has an $\sO$-free $\sO$-structure. We call any such structure an \word{$\sO$-integral structure}.
\end{definition}

\begin{para}
It seems that the existence of integral structures of restricted rational Cherednik algebras has never been considered before. Before we give a sufficient condition for their existence, note that for any $s \in \mrm{Ref}_\Gamma$ the set  
   \[
    \mrm{Che}_\Gamma(s) \dopgleich \lbrace (y_j,x_i)_s \mid i,j \in \lbrack 1,n \rbrack \rbrace \subs K
    \]
   for a $K$-basis $(y_i)_{i=1}^n$ of $V$ with dual basis $(x_i)_{i=1}^n$ is independent of the chosen basis. 
\end{para}

\begin{definition}
 We say that $c \in \fR_\Gamma(K)$ is \word{potentially $\sO$-integral} if $c(s) \mrm{Che}_\Gamma(s) \subs \sO$ for all $s \in \mrm{Ref}_\Gamma$.
\end{definition}

\begin{theorem} \label{res_cher_structure_constants_lemma} 
If there exists a datum $(\biy,\sA,\sB,\sG)$ consisting of a basis $\biy$ of $V$ with dual basis $\bix$, a basis $\sA$ of $K \lbrack V \rbrack_G$, a basis $\sB$ of $K \lbrack V^* \rbrack_G$, and a generating system $\sG$ of $G$ satisfying all of the following properties, then any potentially $\sO$-integral parameter $c \in \fR_\Gamma(K)$ is already $\sO$-integral:
    \begin{enum_thm}
        
        \item \label{res_cher_structure_constants_lemma:img} $\sA$ contains the images of the elements of $\bix$ in $K \lbrack V \rbrack_G$ and every element of $\sA$ is an $\sO$-linear polynomial in these images. The basis $\sB$ satisfies the analogous conditions.
                
        \item \label{res_cher_structure_constants_lemma:str} The structure constants of $K \lbrack V \rbrack_G$ with respect to $\sA$ are contained in $\sO$. The structure constants of $K \lbrack V^* \rbrack_G$ with respect to $\sB$ satisfy the analogous conditions.
        
        \item \label{res_cher_structure_constants_lemma:G} For all $g \in \sG$ the action of $g$ on $V$ in the basis $\biy$ and the action of $g$ on $V^*$ in the basis $\bix$ is described by matrices with entries in $\sO \subs K$. 

    \end{enum_thm}
\end{theorem}

\begin{proof}
Let $\bix = (x_i)_{i=1}^n$ and $\biy = (y_i)_{i=1}^n$. Let $\ol{x}_i$ and $\ol{y}_i$ denote the images of $x_i$ and $y_i$ in $K \lbrack V \rbrack_G$ and $K \lbrack V^* \rbrack_G$, respectively. 
A $K$-basis of $\ol{\rH}_c$ is given by $(abg)_{a \in \sA, b \in \sB, g \in G}$ and  it suffices to show that the structure constants of $\ol{\rH}_c$ with respect to this basis are contained in $\sO$. Due to \ref{res_cher_structure_constants_lemma:str}, products of the form $aa'$ and $bb'$ with $a,a' \in \sA$ and $b,b' \in \sB$ are $\sO$-linear combinations of elements of $\sA$ and $\sB$, respectively. Let $g \in \sG$. Then by \ref{res_cher_structure_constants_lemma:G} we have $\,^gx_i = \sum_{j=1}^n \alpha_{ij} x_j$ with $\alpha_{ij} \in \sO$. Since the $\ol{x}_i$ are contained in $\sA$ by \ref{res_cher_structure_constants_lemma:img}, it follows that $\,^g\ol{x}_i = \sum_{j=1}^n \alpha_{ij} \ol{x}_j$ is the basis expansion of $\,^g\ol{x}_i$ in the basis $\sA$. Hence, the structure constants of the action of $G$ on the elements of $\ol{\bix} \dopgleich (\ol{x}_i)_{i=1}^n$ are contained in $\sO$. If $\lambda \in \bbN^n$, then
\[
\,^g \ol{\bix}^\lambda  = \,^g\left( \prod_{i=1}^n \ol{x}_i^{\lambda_i} \right) = \prod_{i=1}^n \,^g \ol{x}_i^{\lambda_i}
\] 
By what we have just said, the elements $\,^g \ol{x}_i$ are $\sO$-linear combinations of the elements of $\bix$. It now follows from \ref{res_cher_structure_constants_lemma:str} that $\,^g \ol{\bix}^\lambda$ is an $\sO$-linear combination of the elements of $\sA$. This extends to the action of $G$ on all elements of $K \lbrack V \rbrack_G$ and therefore the structure constants of the multiplication of elements of $K \lbrack V \rbrack_G \subs \ol{\rH}_c$ with group elements are also contained in $\sO$. Analogously, this holds for the action of $G$ on $K \lbrack V^* \rbrack_G$. The only products of basis elements not already covered are those of the form $ba$ for $b \in \sB$ and $a \in \sA$. We have 
\[
\ol{y}_j \ol{x}_i = \ol{x}_i \ol{y}_j + \sum_{s \in \mrm{Ref}_\Gamma} (y_j,x_i) c(s)s 
\]
and this is an $\sO$-linear combination of basis elements.  By a recursive application of this and the fact that all other basis elements of $\sA$ and $\sB$ are polynomials in the $\ol{x}_i$ and the $\ol{y}_i$, respectively, we see that all the products $ba$ are $\sO$-linear combination of basis elements. This shows that $\ol{\rH}_c$ has an $\sO$-free $\sO$-structure.
\end{proof}

\begin{proposition} \label{rrca_integral_structure_exists}
For any basis $\biy$ of $V$ there is a basis $\sA$ of $K \lbrack V \rbrack_G$ and a basis $\sB$ of $K \lbrack V^* \rbrack_G$ satisfying Theorem \ref{res_cher_structure_constants_lemma}\ref{res_cher_structure_constants_lemma:img}.
\end{proposition}

\begin{proof}
Let $\bix = (x_i)_{i=1}^n$. We can then write $K \lbrack V \rbrack = K \lbrack x_1,\ldots,x_n \rbrack$. %
Let $\bif$ be a system of fundamental invariants of $\Gamma$. Note that the degrees of the elements of $\bif$ are strictly greater than $1$, since if $f \in \bif$ would be of degree equal to $1$, then it would be an element of $V^*$ fixed by $G$ and thus equal to zero as $\Gamma^*$ is essential by assumption. Since the Hilbert ideal $\fh_\Gamma$ of $\Gamma$ is the homogeneous ideal generated by $\bif$, it follows that $\fh_\Gamma$ does not contain linear polynomials. Now, extend $\bif$ to a Gröbner basis $\wt{\bif}$ of the Hilbert ideal $\fh_\Gamma$ of $\Gamma$ with respect to the lexicographical order. A monomial basis $\sA$ of $K \lbrack V \rbrack_G$ is then given by the images of the elements    
\[
\lbrace \bix^\alpha \mid \alpha \in \bbN^n \tn{ and } \bix^\alpha \tn{ is not divisible by some } \mrm{LT}(f) \tn{ for } f \in \sF \rbrace 
\]
in $K \lbrack V \rbrack_G$ (see Example \ref{example_groebner_basis_normal_form}). Now, suppose that the image of $x_i$ in $K \lbrack V \rbrack_G$ would not be contained in $\sA$. Then by definition there exists $f \in \wt{\bif}$ such that $\mrm{LT}(f)$ divides $x_i$. But this means that $f$ is a linear polynomial and we just argued that no linear polynomial is contained in the Hilbert ideal, so this is not possible. We can apply the same arguments to $K \lbrack V^* \rbrack_G$ and this proves the claim.
\end{proof}

\begin{proposition} \label{rrca_integral_structure_dedekind}
For all but finitely many maximal ideals $\fm$ of $\sO$ any potentially $\sO_\fm$-integral parameter $c \in \fR_\Gamma(K)$ is $\sO_\fm$-integral. We call those $\fm$ for which this is true \word{good} for the restricted rational Cherednik algebras of $\Gamma$.
\end{proposition}

\begin{proof}
Let $\biy$ be a basis of $V$. We know from Proposition \ref{rrca_integral_structure_exists} that  we can find $\sA$ and $\sB$ satisfying Theorem \ref{res_cher_structure_constants_lemma}\ref{res_cher_structure_constants_lemma:img}. Since everything is finite-dimensional, the set $S$ of the structure constants occurring in Theorem \ref{res_cher_structure_constants_lemma}\ref{res_cher_structure_constants_lemma:str} and Theorem\ref{res_cher_structure_constants_lemma}\ref{res_cher_structure_constants_lemma:G} is finite. Since $\sO$ is a Dedekind domain, we have $S \subs \sO_\fm$ for all but finitely many maximal ideals $\fm$ of $\sO$ and so the assumptions in Theorem \ref{res_cher_structure_constants_lemma} are satisfied for the bases $\sA$ and $\sB$, and the ring $\sO_\fm$.
\end{proof}

The proof of Proposition \ref{rrca_integral_structure_exists} and Proposition \ref{rrca_integral_structure_dedekind} gives us an explicit way to find good maximal ideals of $\sO$. This is summarized in Algorithm \ref{rrca_good_maximal_ideals_algo}.
\begin{algorithm}[htbp] 
\caption{Find good maximal ideals} \label{rrca_good_maximal_ideals_algo}
Choose an explicit realization $\biG \subs \GL_n(K)$ of $G$. This amounts to choosing a basis $\biy$ of $V$. Let $\bix$ be the dual basis. \\

Compute fundamental invariants $\bif$ of $\biG$ and $\bif^*$ of the dual group $\biG^*$. \\

Compute a Gröbner basis of $\fh_\biG = \langle \bif \rangle$ and of $\fh_{\biG^*} = \langle \bif^* \rangle$. \\

Compute monomial bases $\sA$ of the coinvariant algebra $K \lbrack \bix \rbrack/\fh_{\biG}$ and $\sB$ of $K \lbrack \biy \rbrack/\fh_{\biG^*}$ using the Gröbner bases. \\

Compute using the Gröbner bases the structure constants of the coinvariant algebras with respect to the bases $\sA$ and $\sB$, respectively.  \\

Let $S$ be the set of all denominators occurring in these structure constants. \\

Choose a generating system $\sG$ of $G$ and extend $S$ by the denominators occurring in the corresponding matrices and their inverses. \\

Then all $\fm$ not containing any element of $S$ are good. 
\end{algorithm}

Precisely this method is performed by the command \code{BadPrimesForRRCA} in \champ. In \cite[\S22]{Thiel.UOn-restricted-ration} we computed sets of primes which contain all bad primes (for explicit choices of the bases) for the exceptional complex reflection groups $\rG_4$ up to $\rG_{28}$ to ensure correctness of our computations. We remark that some of these primes are surprisingly large and we do not yet have a  theoretical explanation for them. 

\subsection{The generic situation for restricted rational Cherednik algebras} \label{gen_rrca}

The primary case we are considering is the following. Let $K$ be a number field with ring of integers $\sO$ and let $R$ be the polynomial ring over $K$ with indeterminates $(c_s)_{s \in \sC_\Gamma}$. Let $c:\sC_\Gamma \rarr R$ be the obvious map and let $\bic$ be the composition of this map with the embedding into the quotient field of $R$. Let $\ol{\bH} \dopgleich \ol{\rH}_c$ be the \word{generic restricted rational Cherednik algebra} for $\Gamma$. Let $\fm \in \Max(\sO)$ be a good maximal ideal. Then for any $u \in \mrm{Che}_\Gamma^{-1}\sO^{\sC_\Gamma}$ the pair $(\fm,u)$ is a finite field specialization and we have the morphism
\[
\begin{tikzcd}
\rG_0(\ol{\rH}_\bic) \arrow{r}{\rd_{\ol{\bH}}^u} \arrow[swap,bend right]{rr}{\rd_{\ol{\bH}}^{\fm,u}} & \rG_0( \ol{\rH}_u ) \arrow{r}{\rd_{\wt{\rH}_u}^\fm} & \rG_0( \wt{\rH}_u(\fm_\fm))
\end{tikzcd}
\]
where $\wt{\rH}_u$ is some $\sO_\fm$-integral structure of $\ol{\rH}_u$. As explained in Remark \ref{high_chance} the probability of this morphism being trivial in the sense that it induces a bijection between the simple modules is quite high. Thus a random choice of $u$ will bring us in position of employing Proposition \ref{strongly_positive_map_props}. It remains to understand how we can lift back the results from the right to the left in this diagram and this is the topic of the next paragraph. 

Before we go there, we point out that the same idea works of course if instead of a parameter $c$ yielding the generic point of the whole parameter space $\fR_\Gamma$ as above we take a parameter yielding the generic point of some closed subscheme of $\fR_\Gamma$, e.g., some hyperplane. To have this possibility at hand was one of the reasons why we chose a general commutative $K$-algebra as base ring everywhere and why we put emphasis on \champ\ being able to handle general base rings. In exactly this way—starting with the generic situation and then considering restrictions to hyperplanes—we approach the cases $\rG_4$, $\rG_{13}$, and $\rG_{20}$.

\section{Reconstructing submodules from abstract structures} \label{modfinder}

Now that we found a way of transporting modules to an algebra over a finite field we have to figure out how we can lift back the results obtained there to the initial setting. The idea is the following: if the morphism $d$ induced by a finite field specialization as in (\ref{d_of_finite_field_spec}) satisfies the condition in Proposition \ref{strongly_positive_map_props}\ref{strongly_positive_map_props_1}, then we can think of it as not destroying the structure of modules. Hence, the ``abstract structure'' of the radical of the image of a module $V$ with simple head under this morphism should be the same as the one of $V$ itself. From this ``abstract structure'' we might be able to compute a candidate for the radical of $V$ and using the morphism $d$ we can check if this candidate was the correct one. Let us now make precise what we mean by ``abstract structure'' and how the candidate production works.

\subsection{Abstract structures}

Let $V$ be an $n$-dimensional vector space over a field $K$ with basis $\biv$ and let $U$ be an $m$-dimensional subspace. For a basis $\biu$ of $U$ let $\rM_\biu^\biv \in \Mat_{n \times m}(K)$ be the matrix of the embedding $U \hookrightarrow V$ with respect to the chosen bases. The class $\sM_U^\biv$ of $\rM_\biu^\biv$ in $\Mat_{n \times m}(K)/\GL_m(K)$ consists precisely of the matrices $\rM_{\biu'}^\biv$ for bases $\biu'$ of $U$. It is an elementary fact that inside $\sM_U^\biv$ there exists precisely one matrix in reduced column echelon form which we denote by $\rM_U^\biv$. Hence, once we fixed a basis of $V$, the subspaces of $V$ are in bijection with $n \times m$-matrices in reduced column echelon form. We will now define the notion of the abstract structure of $U$ with respect to $\biv$ by using the matrix $\rM_U^\biv$. \\

Let $M \in \mrm{Mat}_{n \times m}(K)$. If $\sE(M)$ denotes the set of entries of $M$ and if $\theta:\sE(M) \rarr S$ is a map into a set $S$, then we denote by $\theta^*(M) \in \mrm{Mat}_{n \times m}(S)$ the matrix defined by $(\theta^*(M))_{ij} \dopgleich \theta(M_{ij})$. We denote by $M_{i,\smallbullet}$ the $i$-th row of $M$ and by $M_{\smallbullet,j}$ the $j$-th column of $M$. 
We define $\mrm{Supp}(M_{i,\smallbullet}) \dopgleich \lbrace j \in \lbrack 1,m \rbrack \mid M_{ij} \neq 0 \rbrace$. Analogously we define $\mrm{Supp}(M_{\smallbullet,j})$ and $\mrm{Supp}(M)$. 

Now, suppose that $M$ is in reduced column echelon form. We define two matrices $_{\mrm{c}}M, {_{\mrm{f}}}M \in \mrm{Mat}_{n \times m}(\bbN_{>0})$ as follows. First, decompose $M$ as $M = {_{\mrm{c}}}M + {_{\mrm{f}'}}M$, where each column of ${_{\mrm{c}}}M$ just consists of the leading entry $1$ of the corresponding column of $M$ (if there is one) and ${_{\mrm{f}'}}M$ is the matrix $M-{_{\mrm{c}}}M$. We call ${_{\mrm{c}}}M$ the \word{coarse structure} of $M$. Let $\sE$ be the set of entries of ${_{\mrm{f}'}}M$ and for $x \in \sE$ let $\sE_x \dopgleich \lbrace (i,j) \in \lbrack 1,n \rbrack \times \lbrack 1,m \rbrack \mid M_{ij} = x \rbrace$. We equip each $\sE_x$ with the lexicographical order, which is a total order so that $\sE_x$ has a unique minimum, and define an order $\leq$ on $\sE$ by $x \leq y$ if and only if $\mrm{min} \ \sE_x \leq \mrm{min} \ \sE_y$. This is a total order on the finite set $\sE$ so that assigning to each $x \in \sE$ its position in $\sE$ relative to $\leq$ defines a function $e:\sE \rarr \bbN_{>0}$. We now define ${_{\mrm{f}}}M \dopgleich e^*({_{\mrm{f}'}}M)$ and call this the \word{fine structure} of $M$. We call the pair $\mrm{Abs}(M) \dopgleich ({_{\mrm{c}}}M,{_{\mrm{f}}}M)$, which we also write as ${_{\mrm{c}}}M+{_{\mrm{f}}}M$, the \word{abstract structure} of $M$ and call $\#\sE$ the \word{complexity} of $M$. By $\mrm{Abs}_{n \times m}$ we denote the set of abstract structures of $n \times m$-matrices in reduced column echelon form.

\begin{example}
Let 
\[
M \dopgleich \begin{pmatrix} 1 & 0 \\ 0 & 1 \\ 2 & 1 \\ 1 & 4 \end{pmatrix} = \underbrace{\begin{pmatrix}  1 & 0 \\ 0 & 1 \\ 0 & 0 \\ 0 & 0 \end{pmatrix}}_{_{\mrm{c}}M} + \underbrace{\begin{pmatrix} 0 & 0 \\ 0 & 0 \\ 2 & 1 \\ 1 & 4 \end{pmatrix}}_{{_{\mrm{f}'}}M} \in \mrm{Mat}_{3 \times 2}(\bbQ) \;.
\]
Then
\[
 \mrm{Abs}(M) = \begin{pmatrix} 1 & 0 \\ 0 & 1 \\ 1 & 2 \\ 2 & 3 \end{pmatrix} = \underbrace{\begin{pmatrix} 1 & 0 \\ 0 & 1 \\ 0 & 0 \\ 0 & 0  \end{pmatrix}}_{_{\mrm{c}}M} + \underbrace{\begin{pmatrix} 0 & 0 \\ 0 & 0 \\ 1 & 2 \\ 2 & 3 \end{pmatrix}}_{_{\mrm{f}}M} \in \mrm{Mat}_{3 \times 2}(\bbN_{>0}).
\]
In this example we have $\sE = \lbrace 2,1,4 \rbrace$ and $e:\sE \rarr \lbrack 1,3 \rbrack$ is defined by $e(2) = 1$, $e(1) = 2$, $e(4) = 3$. The complexity of $M$ is equal to $3$.
\end{example}

\begin{definition}
If $V$ is a finite-dimensional vector space over a field $K$ with basis $\biv$, then the \word{abstract structure} $\mrm{Abs}^\biv_U$ with respect to $\biv$ of a subspace $U$ of $V$ is the abstract structure of the matrix $\rM_U^\biv$.
\end{definition}

\begin{definition}
If an abstract structure $M \dopgleich (_\rc M, {_\rf M}) \in \mrm{Abs}_{n \times m}$ with $m \leq n$ is given, then for any map $\theta:\sE(_\rf M) \rarr K^\times$ with $\theta(i) \neq \theta(j)$ for $i \neq j$ we get a matrix $_\rc M + \theta^*{_\rf M} \in \Mat_{n \times m}(K)$ in reduced column echelon form describing a unique subspace $\rU_{M,\theta}^\biv$ of $V$ with respect to the basis $\biv$. We call this subspace the \word{concretization} of $M$ with respect to $\theta$ and $\biv$.  
\end{definition}

\begin{para}
Note that an abstract structure itself is \textit{independent} of a base field—this is precisely the point of abstract structures. 
\end{para}

\subsection{Existence of submodules with prescribed abstract structure}

We can now formulate the primary aim of this paragraph and we do this in a graded setting as the efficiency of \champ\ also relies on the fact that we make use of gradings throughout.

\begin{question} \label{modfinder_question}
Let $A$ be a finite-dimensional $\bbZ$-graded algebra over a field $K$, let $V$ be a $\bbZ$-graded $n$-dimensional $A$-module, and let $\biv \dopgleich (v_i)_{i=1}^n$ be a homogeneous basis of $V$. The question this whole paragraph is about is: 
\begin{quotation}
Given an abstract structure $M \dopgleich  (_\rc M, {_\rf M}) \in \mrm{Abs}_{n \times m}$ with $m \leq n$, is there a graded submodule $U$ of $V$ with $\mrm{Abs}_U^\biv = M$? In other words, is there a map $\theta:\sE(_\rf M) \rarr K^\times$ with $\theta(i) \neq \theta(j)$ for $i \neq j$ such that the concretization $\rU_{M,\theta}^\biv$ is a graded submodule of $V$?
\end{quotation}

\end{question}

To analyze this question we choose a set $\bia \dopgleich (a_k)_{k=1}^r$ of homogeneous $K$-algebra generators of $A$ and denote for each $k \in \lbrack 1,r \rbrack$ by $X^{(k)} \in \Mat_n(K)$ the matrix describing the action of $a_k$ on $V$ in the basis $\biv$, i.e., 
\begin{equation}  \label{A_action}
a_k v_i = \sum_{l=1}^n X_{li}^{(k)} v_l = \sum_{l \in D_{ki}^\rr} X_{li}^{(k)} v_l 
\end{equation}
for all $j \in \lbrack 1,n \rbrack$, where 
\[ 
D_{ki}^\rr \dopgleich \lbrace l \in \lbrack 1 ,n \rbrack \mid \mrm{deg}(a_k) + \mrm{deg}(v_i) = \mrm{deg}(v_l) \rbrace \;.
\]

\begin{theorem} \label{module_abs_existence}
The answer to Question \ref{modfinder_question} is positive if and only if the following conditions are satisfied:
\begin{enum_thm}
\item \label{module_abs_existence:deg} For each $j \in \lbrack 1,m \rbrack$ the degree of $v_i$ is constant for all $i \in \mrm{Supp}(M_{\smallbullet,j})$. We define $d_M^\rc(j)$ to be this degree.
\item There exist pairwise different $\theta_1,\ldots,\theta_s \in K^\times$, where $s$ is the complexity of $M$, and a family 
\[
(Y^{(k,j)}_l)_{\substack{ k \in \lbrack 1,r \rbrack, j \in \lbrack 1,m \rbrack \\ l \in D_{kj}^\rc } } \subs K \;,
\]
where 
\[
D_{kj}^\rc \dopgleich (d_M^\rc)^{-1}(\mrm{deg}(a_k) + d_M^\rc(j)) 
\]
such that the equations
\begin{equation} \label{E1_equ}
E^1_{i,j,k}: \quad \sum_{l \in I_{ijk}} {_\rc M_{lj} X_{li}^{(k)} } + \sum_{l \in I_{ijk}} \theta_{_\rf M_{lj}} X_{il}^{(k)} = 0 
\end{equation}
hold for all $j \in \lbrack 1,m \rbrack$, $k \in \lbrack 1,r \rbrack$, $i \in \mrm{Supp}(M_{\smallbullet,j})$, and such that the equations
\begin{equation} \label{E2_equ}
E_{ijk}^2: \quad \sum_{l \in I_{ijk}} {_\rc M_{lj}} X_{il}^{(k)} + \sum_{l \in I_{ijk}} \theta_{_\rf M_{lj}} X_{il}^{(k)} = \sum_{l \in D_{kj}^\rc} Y_l^{(k,j)} {_\rc M_{il}} +  \sum_{l \in D_{kj}^\rc} Y_l^{(k,j)} \theta_{_\rf M_{il}} 
\end{equation}
hold for all $j \in \lbrack 1,m \rbrack$, $k \in \lbrack 1,r \rbrack$, and $i \in \lbrack 1,n \rbrack \setminus \mrm{Supp}(M_{\smallbullet,j})$, where
\[
I_{ijk} \dopgleich \lbrace l \in \mrm{Supp}(M_{\smallbullet,j}) \mid i \in D_{kl}^\rr \rbrace \;.
\]
\end{enum_thm}
\end{theorem}

\begin{proof}
Suppose that the conditions are satisfied. Let $\theta:\lbrack 1,s \rbrack \rarr K^\times$ be the map with $\theta(i) \dopgleich \theta_i$. Then the concretization $U \dopgleich \rU_{M,\theta}^\biv$ defines a unique subspace of $V$. Let $N_{\smallbullet, j} \dopgleich {_\rc M_{\smallbullet, j}} + (\theta^*{_\rf M})_{\smallbullet, j}$ be the ``specialization'' of the $j$-th column of $M$ in $\theta$. Define
\begin{equation} \label{uj_expr}
u_j \dopgleich \sum_{i=1}^n N_{i,j} v_i = \sum_{i \in \mrm{Supp}(M_{\smallbullet,j})} N_{i,j} v_i = \sum_{i \in \mrm{Supp}(M_{\smallbullet,j})} {_\rc M_{ij}}v_{i} + \theta_{_\rf M_{ij}} v_i \;.
\end{equation}
Then $(u_j)_{j=1}^m$ is a basis of $U$ and because of \ref{module_abs_existence:deg} this is a graded subspace. It remains to show that $U$ is $A$-invariant. This holds if and only if $a_k U \subs U$ for all $k \in \lbrack 1,r\rbrack$, and this in turn holds if and only if $a_k u_j \in U$ for all $k$ and $j$, so $a_k u_j \in \langle u_1, \ldots, u_m \rangle_K$. As $u_j$ is homogeneous of degree $\mrm{deg}(a_k) + d_M^\rc(j)$, this is equivalent to $a_k u_j \in \langle u_l \mid l \in D_{kj}^\rc \rangle$. This is equivalent to the existence of elements $Y_l^{(k,j)} \in K$ such that 
\begin{equation} \label{akuj_expr}
a_ku_j = \sum_{l \in D_{kj}^\rc} Y_l^{(k,j)} u_l \;.
\end{equation}
Combining equations (\ref{A_action}), (\ref{uj_expr}), and (\ref{akuj_expr}) implies that this is equivalent to the following equality for each $j \in \lbrack 1,m \rbrack$ and $k \in \lbrack 1,r \rbrack$:

\begin{align*}
&a_k \left( \sum_{i \in \mrm{Supp}(M_{\smallbullet,j})} {_\rc M_{ij} v_i} + \theta_{_\rf M_{ij}} v_i \right) = \sum_{l \in D_{kj}^\rc} Y_l^{(k,j)} \left(\sum_{i \in \mrm{Supp}(M_{\smallbullet,j})} {_\rc M_{il} v_i} + \theta_{_\rf M_{il}} v_i \right) \allowdisplaybreaks[4]\\
 \Leftrightarrow & \sum_{i \in \mrm{Supp}(M_{\smallbullet,j})} \sum_{l \in D_{ki}^\rr} {_\rc M_{ij}} X_{li}^{(k)} v_l + \sum_{i \in \mrm{Supp}(M_{\smallbullet,j})} \sum_{l \in D_{ki}^\rr} \theta_{_\rf M_{ij}} X_{li}^{(k)} v_l \allowdisplaybreaks[4]\\ & = \sum_{i \in \mrm{Supp}(M_{\smallbullet,j})} \sum_{l \in D_{kj}^\rc} Y_l^{(k,j)} {_\rc M_{il}} v_i +  \sum_{i \in \mrm{Supp}(M_{\smallbullet,j})} \sum_{l \in D_{kj}^\rc} Y_l^{(k,j)} \theta_{_\rf M_{il}} v_i \allowdisplaybreaks[0]\\
 \Leftrightarrow & \sum_{i=1}^n \sum_{l \in I_{ijk}} {_\rc M_{lj}} X_{il}^{(k)} v_i + \sum_{i=1}^n \sum_{l \in I_{ijk}} \theta_{_\rf M_{lj}} X_{il}^{(k)} v_i \\ & = \sum_{i \in \mrm{Supp}(M_{\smallbullet,j})} \sum_{l \in D_{kj}^\rc} Y_l^{(k,j)} {_\rc M_{il}} v_i +  \sum_{i \in \mrm{Supp}(M_{\smallbullet,j})} \sum_{l \in D_{kj}^\rc} Y_l^{(k,j)} \theta_{_\rf M_{il}} v_i \;.
\end{align*}
As $\biv$ is a basis of $V$, each of these equations holds if and only if the coefficients of $v_i$ for each $i \in \lbrack 1,n \rbrack$ are the same. If $i \notin \mrm{Supp}(M_{\smallbullet,j})$ the coefficient equation is
\[
\sum_{l \in I_{ijk}} {_\rc M_{lj}} X_{il}^{(k)} + \sum_{l \in I_{ijk}} \theta_{_\rf M_{lj}} X_{il}^{(k)}  = 0 \;.
\]
If $i \in \lbrack 1,n \rbrack \setminus \mrm{Supp}(M_{\smallbullet,j})$ the coefficient equation is
\[
\sum_{l \in I_{ijk}} {_\rc M_{lj}} X_{il}^{(k)} + \sum_{l \in I_{ijk}} \theta_{_\rf M_{lj}} X_{il}^{(k)} = \sum_{l \in D_{kj}^\rc} Y_l^{(k,j)} {_\rc M_{il}} +  \sum_{l \in D_{kj}^\rc} Y_l^{(k,j)} \theta_{_\rf M_{il}} \;.
\]
These are the two asserted types of equations. It is evident from the discussion that these equations are also necessary for the existence of a graded submodule.
\end{proof}

\subsection{Finding submodules with prescribed abstract structure (ModFinder)}
Let $E^1_{M,\biv} \dopgleich (E^1_{i,j,k})$ be the system of equations defined by (\ref{E1_equ}), let $E^2_{M,\biv} \dopgleich (E^2_{i,j,k})$ be the system of equations defined by (\ref{E2_equ}), and let $E_{M,\biv}$ be the whole system. For finding a graded submodule of $V$ with abstract structure $M$ we have to solve the system $E_{M,\biv}$ for the \textit{$\theta$-variables} $\theta_1,\ldots,\theta_s$ and the \textit{auxiliary variables} $Y_l^{(k,j)}$. If there is a unique submodule with this abstract structure—for example if $M$ is the abstract structure of the unique maximal submodule when $V$ has simple head—this system will have a unique solution we are searching for. \\

While $E^1_{M,\biv}$ is an inhomogeneous linear system for the $\theta$-variables, the system $E^2_{M,\biv}$ is quadratic because of the products $Y_l^{(k,j)} \theta_{\rf M_{il}}$ occurring in the equations. Hence, it will be very difficult in general to solve this system. But we can still try to consecutively solve \textit{linear parts} of this system. Namely, we can start solving $E_{M,\biv}^1$, which is easy as it is a linear system. The point is now that this system might already pin down one of the $\theta$-variables. When plugging in the determined $\theta$-variables into the system $E^2_{M,\biv}$ we might get further \textit{linear} equations just involving the auxiliary variables. If we can determine some of the auxiliary variables, then plugging them into $E^2_{M,\biv}$ might yield new \textit{linear} equations for the $\theta$-variables which might pin down further $\theta$-variables etc. This means we consecutively solve the ``specialized systems'' $L_{M,\biv}(\theta',Y')$ given by the linear equations of the system $E_{M,\biv}$ when plugging in a family $\theta'$ of $\theta$-variables and a family $Y'$ of auxiliary variables. If this process leads to a (unique) solution of $E_{M,\biv}$ we say that this system is (uniquely) \word{linearly solvable}. It might happen, however, that at some stage we cannot determine any new variables—then the system is not linearly solvable.  \\

As we will work with modules of dimension up to 3,000 we need a very efficient strategy for determining the $\theta$-variables by linear equations of $E_{M,\biv}$ (if this is possible at all). To this end, we define for any $q \in \lbrack 1,s \rbrack$ a subsystem of $L_{M,\biv}^q(\theta',Y')$ just consisting of the linear equations of $E_{M,\biv}(\theta',Y')$ involving $\theta_q$ and all \word{dependent variables}. To make this precise, denote for a subsystem $E$ of $E_{M,\biv}$ by $\Theta(E)$ the set of non-determined $\theta$-variables occurring in these equations. For $q \in \lbrack 1,s \rbrack$ let $\wt{L}_{M,\biv}^{q}(\theta',Y')$ just consist of the equations of $L_{M,\biv}(\theta',Y')$ involving the variable $\theta_q$, i.e., 
\[
\wt{L}_{M,\biv}^{q}(\theta',Y') \dopgleich \lbrace L \in L_{M,\biv}(\theta',Y') \mid \theta_q \in \Theta(L) \rbrace \;.
\]
Now, define $L_{M,\biv}^{q}(\theta',Y')$ inductively as follows. First,  $L_{M,\biv}^{q}(\theta',Y') \dopgleich \wt{L}_{M,\biv}^{q}(\theta',Y') $. For each $\theta_p \in \Theta(L_{M,\biv}^{q}(\theta',Y') )$ we add to $L_{M,\biv}^{q}(\theta',Y') $ the equations of $\wt{L}_{M,\biv}^p(\theta',Y')$. We repeat this process until  $L_{M,\biv}^{q}(\theta',Y')$ stabilizes.

We will split the system $E_{M,\biv}$ once more by defining $L_{M,\biv}^{q,\mbi{g}}(\theta',Y')$ for a subset $\mbi{g} \subs \lbrack 1,r \rbrack$ as the subsystem of $L_{M,\biv}^{q}(\theta',Y')$ just involving equations $E_{ijk}^1$ or $E_{ijk}^2$ with $k \in \mbi{g}$. The idea behind this is that we do not want to consider all algebra generators at once—perhaps a few algebra generators will be sufficient to determine all $\theta$-variables and this means we have to consider fewer equations. This idea turned out to be very efficient in experiments (see \S\ref{experiments}). \\
Our idea of solving $E_{M,\biv}$ is now summarized in Algorithm \ref{modfinder_alg}.
\begin{algorithm}[htbp] 
 \caption{Finding submodules with prescribed abstract structure (\textsc{ModFinder})}
 \label{modfinder_alg}
\KwData{Data as in Question \ref{modfinder_question} and Theorem \ref{module_abs_existence} satisfying Theorem \ref{module_abs_existence}\ref{module_abs_existence:deg}, and a subset $\mbi{g} \subs \lbrack 1,r \rbrack$.}
\KwResult{Decides if the system $E_{M,\biv}$ is uniquely linearly solvable. If so, returns a graded submodule $U$ of $V$ with $\mrm{Abs}_U^\biv = M$.}  

$\theta' \dopgleich \emptyset$; $Y \dopgleich \emptyset$\;

\While{$\# \theta' \neq s$}
{
progress $\dopgleich$ false\;
    \For{$q \in \Theta(E_{M,\biv}(\theta',Y'))$ \label{modfinder_thetaq_choice}} 
    {
        \If{$L_{M,\biv}^{q,\mbi{g}}(\theta',Y') \tn{ is not consistent}$}
        {
        \Return{There is no graded submodule with abstract structure $M$}\;
        }
        \tn{Let $\theta''$ and $Y''$ be the $\theta$-variables and auxiliary variables, respectively, determined by $L_{M,\biv}^{q,\mbi{g}}(\theta',Y') $}\;
        \If{$\theta'' \tn{ or } Y'' \tn{ contains a variable not in } \theta' \tn{ or } Y'\tn{, respectively,}$}
        {
        $\theta' \dopgleich \theta' \cup \theta''$; $Y' \dopgleich Y' \cup Y''$\;
        progress $\dopgleich$ true\;
        }   
    }
    \If{\tn{progress = false}}
    {
        \eIf{$\mbi{g} = \lbrack 1,r \rbrack$}
        {
            \Return{$E_{M,\biv}$ is not uniquely linearly solvable}\;
        }
        {
            \tn{Repeat the above algorithm with $\mbi{g} = \lbrack 1,r \rbrack$}\;
        }
    }
}
\tn{Check if $\rU_{M,\theta}^\biv$ is indeed a submodule of $V$}\; \label{modfinder_submod_check}
\eIf{\tn{this is true}}
{
    \Return{$ \rU_{M,\theta}^\biv$}\;
}
{
    \Return{$E_{M,\biv}$ is not uniquely linearly solvable}\;
}
\end{algorithm}
This algorithm—which we call the \textsc{ModFinder} algorithm—has been implemented in this way (and with several additional ideas we cannot discuss here) in \champ\ in the subpackage \code{ModFinder}. In line \ref{modfinder_submod_check} we have to check whether the concretization $\rU_{M,\theta}^\biv$ is indeed a submodule as we are just solving subsystems of $E_{M,\biv}$ and just verify necessary conditions up to this point. This can efficiently be checked using the \word{graded spinning algorithm}—a graded adaption of the standard spinning algorithm explained for example in \cite[\S1.3]{LP-representations}. All this is provided by the new type \code{ModGr} for graded modules we have implemented in \champ.

\begin{remark}
Obviously, there is no reason why we can solve $E_{M,\biv}$ just by consecutively solving specialized linear subsystems. For the radicals of Verma modules for restricted rational Cherednik algebras, however, this surprisingly turned out to be almost always the case and our algorithm was amazingly efficient—we cannot yet give theoretical arguments in favor of this. 
\end{remark}

\begin{remark}
In experiments we observed that the choice of $\mbi{g}$ and the order in which we try to determine $\theta$-variables (line \ref{modfinder_thetaq_choice} in Algorithm \ref{modfinder_alg}) can have a serious impact on the runtime of the algorithm (see \S\ref{experiments}). We do not know yet how to determine an optimal choice of $\mbi{g}$ and on the sequence of $\theta$-variables to solve for. The interaction between the subsystems $L_{M,\biv}^{q,\mbi{g}}(\theta',Y')$ seems to be very hard to understand. In \champ\ we have implemented a selection process for the systems which performs quite well in experiments.  
\end{remark}

\section{A Las Vegas algorithm for computing heads and constituents}
\label{heads}

With the theory of finite field specializations and the \textsc{ModFinder} algorithm we can now turn our idea explained abstractly in Strategy \ref{computing_const_abstract} into an algorithm. The result is Algorithm \ref{dec_matrix_algorithm}. Remember that we are considering a finite-dimensional algebra $A$ over a field and a family $(V_\lambda)_{\lambda \in \Lambda}$ of finite-dimensional $A$-modules with simple heads $(S_\lambda)_{\lambda \in \Lambda}$ such that this family is constituent-closed, meaning that every constituent of a member $V_\lambda$ of this family is the head $S_\mu$ of some $V_\mu$. Algorithm \ref{dec_matrix_algorithm} attempts to compute the simple modules $S_\lambda$ and the multiplicities of $S_\mu$ in $V_\lambda$.
\begin{algorithm} 
\caption{Computing heads and decomposition matrices} \label{dec_matrix_algorithm}
    \KwData{Data as explained in \S\ref{heads}}
\KwResult{If successful, returns the simple modules $S_\lambda$ and the multiplicity $m_{\lambda,\mu}$ of $S_\mu$ in $V_\lambda$.}
    \tn{Randomly choose a strongly positive morphism } $d:\rG_0(A) \rarr \rG_0(B)$ \tn{ with } $B$ \tn{ a finite-dimensional algebra over a finite field}\; \label{dec_matrix_alg_random}
    \For{$\lambda \in \Lambda$}
    {
        \tn{Compute a representative $\ol{V}_\lambda$ of $d(\lbrack V \rbrack)$} \;
        \tn{Compute using the \textsc{MeatAxe} the radical } $\ol{J}_\lambda$ \tn{ of } $\ol{V}_\lambda$\;
        \tn{Determine the abstract structure } $\ol{J}_\lambda^{\mrm{abs}}$ \tn{ of } $J$ \tn{ in } $\ol{V}_\lambda$\;
        \tn{Using algorithm \ref{modfinder_alg} try to find a submodule } $J_\lambda$ \tn{ of } $V_\lambda$ \tn{ with abstract structure } $\ol{J}_\lambda^{\mrm{abs}}$\;
        \uIf{$J_\lambda$\tn{ could not be determined}}
        {
       \Return{No success}\;
        }
        \Else
        {
            $Q_\lambda \dopgleich V_\lambda/J_\lambda$ \;
            \tn{Compute a representative } $\ol{Q}_\lambda$ \tn{ of } $d(\lbrack Q_\lambda \rbrack)$\;
            \tn{Check using the \textsc{MeatAxe} if } $\ol{Q}_\lambda$ \tn{ is irreducible}\;
            \If{\tn{this is not true}}
            {
                \Return{No success}\;
            }
            
        }
    }    
    \For{$\lambda \in \Lambda$}
    {
        \tn{Compute using the \textsc{MeatAxe} the constituents } $(\ol{U}_{\lambda,\theta})_{\theta \in \Theta_\lambda}$ \tn{ and their multiplicities } $m_{\lambda,\theta}$ \tn{ of } $\ol{V}_\lambda$\;
        \tn{Find using the \textsc{MeatAxe} an injection } $\iota_\lambda:\Theta_\lambda \hookrightarrow \Lambda$ \tn{ such that } $\ol{U}_{\lambda,\theta} \cong \ol{Q}_{\mu}$ \tn{for } $\mu \in \Lambda$ \tn{ and } $\theta \in \Theta_\lambda$ \tn{ if and only if } $\mu = \iota_\lambda(\theta)$\;
        \If{\tn{no such injection exists}}
        {
            \Return{No success}\;
        }
        $m_\lambda,\iota_\lambda(\theta) \dopgleich m_{\lambda,\theta}$ for all $\theta \in \Theta_\lambda$ \tn{ and } $m_{\lambda,\mu} \dopgleich 0$ \tn{ for all } $\mu \notin \Im \iota_\lambda$\;        
    }

    \Return{$(Q_\lambda)_{\lambda \in \Lambda}, \ (m_{\lambda,\mu})_{\lambda,\mu \in \Lambda}$}\;
\end{algorithm}

We see that there are three branches in our algorithm whose result will be that the algorithm is not successful. On the other hand, if the algorithm is successful, it follows from our discussion that the result returned is the correct result. This means that our algorithm is a so-called Las Vegas algorithm, like the \textsc{MeatAxe} itself. Because of this it is not easy to provide a complexity analysis of our approach. Note that whenever the algorithm is unsuccessful, it makes sense to run it again with a different finite field specialization.

\begin{remark}
In \champ\ we have implemented an extension of the above algorithm motivated by the few cases where it was not successful. Namely in this case we randomly pick a vector $v \in V_\lambda$ and compute (using the graded spinning algorithm) the submodule $U$ of $V_\lambda$ it generates. In case it is a proper submodule,  we compute the quotient $Q \dopgleich V_\lambda/U$ and apply our algorithm to $Q$. If it is again not successful, we repeat this process. With this simple extension we could indeed obtain all results for restricted rational Cherednik algebras we could compute so far.
\end{remark}

\subsection{Application to Gordon's questions}

Let us discuss how we apply our algorithm to Gordon's questions \S\ref{gordon_problems} in case of generic restricted rational Cherednik algebras (see \S\ref{gen_rrca}) for irreducible complex reflection groups. First, we choose a realization $\Gamma$ of the reflection group over a number field $K$ with ring of integers $\sO$ (this is always possible). Then we compute which maximal ideals of $\sO$ are certainly good using Algorithm  \ref{rrca_good_maximal_ideals_algo}. Next, we compute the generic Euler families $\mrm{Eu}_\bic$ (see \S\ref{euler_families}). For each Euler family $\Lambda$ the Verma modules $(\Delta_\bic(\lambda))_{\lambda \in \Lambda}$ form a constituent closed family of modules with simple head to which we apply our algorithm. 

The random finite field specialization (line \ref{dec_matrix_alg_random} of  Algorithm \ref{dec_matrix_algorithm}) is chosen as $\rd_{\ol{\bH}}^{\fm,u}$ by randomly choosing a good maximal ideal $\fm$ and a point $u \in \mrm{Che}_\Gamma^{-1}\sO^{\sC_\Gamma}$ as explained in \S\ref{gen_rrca}. All this is automatically performed in \champ\ by the commands \code{HeadOfLocalModule} and \code{HeadsOfLocalModules} contained in the subpackage \code{RadicalLift}. This command is in general applicable to any constituent closed family of  modules with simple head over an algebra over a rational function field over a number field. Note, however, that one has to ensure by theory that the chosen data $(\fm,u)$ is indeed a finite field specialization in the sense of \S\ref{ffspecs}.

If successful, our algorithm computes the generic Verma families (see \S\ref{verma_families}) and due to the result by Bonnafé–Rouquier (see \S\ref{verma_families}) we also know the Calogero–Moser families. Note that in case of success we have also explicitly computed the simple modules so that we know their dimension, their Poincaré series, and using character theory we can also compute their structure as graded $G$-modules. In this way we can answer all of Gordon's questions.

The same idea is of course applicable if we do not start with the generic algebra $\ol{\bH}$ but with its restriction to a hyperplane, say. This is exactly what we did to get the results in for $\rG_4$, $\rG_{13}$, and $\rG_{22}$.

\begin{remark}
If we work with a generic restricted rational Cherednik algebra $\ol{\bH}$ for a reflection group $\Gamma$ over a \textit{finite} field $K$ which splits over $K$, the choice of the morphism $d$ in line \ref{dec_matrix_alg_random} of the algorithm is actually simpler. As restricted rational Cherednik algebras split, we have a decomposition morphism $\rd_{\ol{\bH}}^\fp: \rG_0(\ol{\bH}(0)) \rarr \rG_0(\ol{\bH}(\fp))$ for any prime ideal $\fp$ of the base ring of $\ol{\bH}$ and we can choose for $\fp$ any $K$-point of $\fR_\Gamma$. This approach is also covered by \champ\ and it applies in particular to Verma modules for rational Cherednik algebras at $t=1$ in positive characteristic (see \cite{BelMar-On-the-smoothness-of-cent-0}).
\end{remark}

\section{Summary of the results} \label{further_results}

\begin{parani}
We summarize here as theorems the results we could get so far using \champ. All results are listed explicitly in tabular form in the ancillary document of this article. We also comment on some observations in the hope that some general theorem lies behind them. The reader should check the website \url{http://thielul.github.io/CHAMP/} and \cite{Thi-CHAMP:-A-Cherednik-A2014} for further results obtained after publication of this article. 
\end{parani}

\begin{theorem} \label{theorem_generic}
For \textit{generic} parameters for the groups 
\[
\rG_{4}, \rG_{5}, \rG_6,\rG_7,\rG_8,\rG_9,\rG_{10},\rG_{12},\rG_{13},\rG_{14},\rG_{15}, \rG_{16}, \rG_{20},\rG_{22}, \rG_{23}=\rH_3, \rG_{24} 
\]
the following holds:
\begin{enum_thm}
\item We have the explicit answers to all of Gordon's questions.
\item Martino's generic parameter conjecture holds.
\item The Calogero–Moser families are equal to the Euler families. This implies that the locus of ``exceptional'' parameters, i.e., those parameters for which the Calogero--Moser families become coarser than the generic Calogero–Moser families, is contained in the Euler variety and is thus a union of hyperplanes.
\item \label{theorem_generic_palindromic} The Poincaré series of the simple modules are palindromic, i.e., their list of coefficients can be reversed without changing the polynomial. \qed
\end{enum_thm}
\end{theorem}

\begin{theorem} \label{theorem_special}
For \textit{all} parameters for the groups 
\[
\rG_{4}, \rG_{12}, \rG_{13}, \rG_{20}, \rG_{22}, \rG_{23}=\rH_3, \rG_{24} 
\]
the following holds\footnote{Note that there is just one parameter for $\rG_{12}$, $\rG_{22}$, $\rG_{23}$, and $\rG_{24}$ so these results are just the generic ones. But for $\rG_4$,  $\rG_{13}$, and $\rG_{20}$ there are two parameters and here much more work has to be done.}:
\begin{enum_thm}
\item We have the explicit answers to all of Gordon's questions.
\item Martino's conjecture holds in its complete form, i.e., the Rouquier $k^\sharp$-families refine the Calogero–Moser $k$-families for all parameters $k$. \qed
\end{enum_thm}
\end{theorem}

\begin{theorem} \label{theorem_dec_matrix}
In all the cases covered by theorems \ref{theorem_generic} and \ref{theorem_special} the following property holds: if $\lambda$ is a character of minimal degree $d$ in a Calogero–Moser family $\sF$, then the multiplicity of $\rL(\mu)$ in $\Delta(\lambda)$ for $\mu \in \sF$ is a positive multiple of $d$.  \qed
\end{theorem}

\begin{theorem}
For the groups
\[
\rG_{4}, \rG_6, \rG_8, \rG_{12}, \rG_{13}, \rG_{14}, \rG_{20}, \rG_{22}, \rG_{23}=\rH_3, \rG_{24} 
\]
we explicitly know the locus of ``exceptional'' parameters. Except for the group $\rG_8$ it coincides precisely with the union of Chlouveraki's essential hyperplanes for cyclotomic Hecke algebras \cite{Chl09-Blocks-and-famil}. For $\rG_8$, however, the Euler hyperplane $k_{1,1} - k_{1,2} + k_{1,3} $ is one additional ``exceptional'' non-essential hyperplane.\footnote{This was first discovered by Bonnafé using different methods.} \qed
\end{theorem}

\begin{theorem} \label{generic_euler_results}
Also for $\rG_6$, $\rG_8$, and $\rG_{14}$ we have the answers to all of Gordon's questions for the generic points of all Euler hyperplanes. \qed
\end{theorem}

\begin{remark}
Theorem \ref{generic_euler_results} does not yet imply that we know the results for \textit{all} parameters for $\rG_6$, $\rG_8$, and $\rG_{14}$ as the parameter space for these groups is three-dimensional and there is no theory of ``semi-continuity'' of the representation theory of restricted rational Cherednik algebras so far. To this end, we would also have to consider all intersections of the Euler hyperplanes—and this would be way too much to compute and document. So, to solve these cases we need new theory.
\end{remark}

\begin{question} \label{questions}
 Our results suggest the following questions:
 \begin{enum_thm}
 \item \label{questions_palindromic} Are the Poincaré series of simple modules for generic parameters always palindromic? If not, what lies behind this property?
 \item Is the property about the decomposition matrices of the Verma modules in theorem \ref{theorem_dec_matrix} always true?
 \item Is the locus of ``exceptional parameters'' always a union of hyperplanes (this was already asked by Bonnafé–Rouquier \cite{Bonnafe.C;Rouquier.R13Cellules-de-Calogero})? Does it always contain the union of Chlouveraki's essential hyperplanes?
 \end{enum_thm}
\end{question}

\begin{remark}
For special parameters it is no longer true that the Poincaré series of simple modules is palindromic. Already for $\rG_4$ on the hyperplane $k_{1,1}-2k_{1,2}=0$ we find a simple module with Poincaré series $1+2t$, which is not palindromic. There are many more counter-examples. 
\end{remark}

\begin{remark}
The first examples we found where the Rouquier families are strictly finer than the Calogero–Moser families are for $\rG_{20}$ and the hyperplanes $k_{1,1}=0$, $k_{1,2}=0$, and $k_{1,1}-k_{1,2}=0$.
\end{remark}

\begin{remark}
So far we have no idea about general properties of the (graded) $G$-module structures of the simple modules. We hope that our explicit results help to reveal them.
\end{remark}

\begin{remark}
We discussed rational Cherednik algebras for reflection groups over arbitrary fields as long as all reflections are diagonalizable and designed \champ\ to work in this generality. In \cite{Thiel.UOn-restricted-ration} we computed for example the representation theory of the restricted rational Cherednik algebra attached to the general orthogonal group $\mrm{GO}_3(3)$ and to modular reflection representations of some symmetric groups. These cases are not yet understood theoretically and we hope that such examples will help to develop a general theory.
\end{remark}

\section{CHAMP} \label{champ}

\begin{parani}
Now, we pass to the experimental part of this article. Everything we discussed so far has been implemented in \champ. The source code and documentation (including a Wiki) of \champ\ is freely available at \url{http://thielul.github.io/CHAMP/}. All parts are licensed under the GPL. Due to some operating system functions used in \champ, it will not work on Windows systems, just on Linux and Mac OS X systems. Moreover, a \textsc{Magma} version of at least 2.19 (released in December 2012) is necessary as we make use of user-defined types which did not exist in earlier versions. 

\subsection{Running Champ}
Once the downloaded package is unpacked one has to configure \champ\ by running
\begin{lstlisting}
$ ./configure
\end{lstlisting}
in a terminal and inside the directory of \champ. This sets several variables to the absolute path of \champ\ and is necessary for working with it. \champ\ is now started by running:
\begin{lstlisting}
$ ./champ 
Loading file "/CHAMP/CHAMP.m"

CHAMP (CHerednik Algebra Magma Package)
Version v1.5
Copyright (C) 2013, 2014 Ulrich Thiel
Licensed under GNU GPLv3, see LICENSE.txt
thiel@mathematik.uni-stuttgart.de
http://thielul.github.io/CHAMP/

>
\end{lstlisting}
Before we give a rough description of the capabilities of \champ, we point out the following important aspect: \vspace{4pt}

\begin{leftbar}
All actions in \textsc{Magma} are \textit{right} actions. This means whenever we start with a reflection group acting from the left and we consider left modules over rational Cherednik algebras, we have to transpose all matrices in \textsc{Magma}. Moreover, the rational Cherednik algebra implemented in \champ\ is the \textit{opposite} algebra of the one we are describing here theoretically. Hence, we have to reverse all products when passing between theory and \champ.
\end{leftbar}

\noindent This reversion process between theory and \champ\ might be confusing at first but we found it much more confusing when artificially working with left actions in \textsc{Magma}.
\end{parani}

\subsection{Reflection groups}

As one aim of \champ\ was to verify Martino's conjecture we had to make sure that we use the same labelings of irreducible characters of complex reflection groups as the one used by Chlouveraki \cite{Chl09-Blocks-and-famil} for the computation of Rouquier families. This is why we imported all relevant data from \textsc{Chevie} (see \cite{CHEVIE-JM-4}) and implemented basic data base support in \champ\ to deal with this data. This is illustrated by the following example:
\begin{lstlisting}
> G:=ExceptionalComplexReflectionGroup(4); 
> CharacterTable(~G);
> G`CharacterNames;
[ \phi_{1,0}, \phi_{1,4}, \phi_{1,8}, \phi_{2,5}, \phi_{2,3}, \phi_{2,1},
  \phi_{3,2} ]
\end{lstlisting}
In this example we loaded the exceptional complex reflection group $\rG_{4}$. The realization is the same as in \textsc{Chevie}, but note that all matrices are transposed. Then we attached the character table to this group. When doing this the names of the characters used in \textsc{Chevie} are automatically loaded and stored in the attribute \code{CharacterNames} of the group. We see in this example that one philosophy of \champ\ is to work with procedures taking a reference to an object as input and store their result in the corresponding attribute of the objects. The reason for this is that we want to have easy access to all data already computed and to handle the large amount of data necessary to work with rational Cherednik algebras. The absolutely irreducible characteristic zero representations are now attached using the procedure \code{Representations(\~{}G,0)} and can be accessed via \code{G`Representations[0]}. Again we use the exact same realizations of these representations as in \textsc{Chevie}. Absolutely irreducible representations in characteristic $p$ can be attached by calling the above command with $p$ instead of $0$. \\
Next to the characters and representations, the reflections are important. A structured collection of the reflections is attached by the command \code{ReflectionLibrary} which gathers all the reflections of a reflection group $\Gamma$ in a nested list of the form
\[
\left( \left( \left(s\right)_{\rH_s = H} \right)_{H \in \Omega} \right)_{\Omega \in \sA_\Gamma} \;.
\]
Hence, for each orbit $\Omega$ of reflection hyperplanes of $\Gamma$ we have for each $H \in \Omega$ a list consisting of the reflections with hyperplane $H$. This allows us to label a reflection of $\Gamma$ by a triple $(i,j,k)$, where $i$ refers to the $i$-th reflection hyperplane orbit, $j$ refers to the $j$-th hyperplane in the orbit labeled by $i$, and $k$ refers to the $k$-th reflection with hyperplane $j$. This is precisely the triple we get when passing a reflection to the function \code{ReflectionID}. From the reflection library we automatically store representatives of the conjugacy classes of reflections in the attribute \code{ReflectionClasses}. 

\subsection{Cherednik algebras}

A generic Cherednik parameter can be obtained as follows:
\begin{lstlisting}
> G:=ExceptionalComplexReflectionGroup(4); 
> c:=CherednikParameter(G : Type:="GGOR"); c;
Mapping from: { 1 .. 2 } to Multivariate rational function field of 
rank 2 over Cyclotomic Field of order 3 and degree 2
    <1, (-zeta_3 + 1)*k_{1,1} + (2*zeta_3 + 1)*k_{1,2}>
    <2, (zeta_3 + 2)*k_{1,1} + (-2*zeta_3 - 1)*k_{1,2}>
\end{lstlisting}
This will be a map $c:\lbrack 1, N \rbrack \rarr L$, where $N$ is the number of conjugacy classes of reflections and $L$ is the appropriate rational function field (the residue field in the generic point of $\fR_\Gamma$). The numbers $1$ to $N$ of the domain of $c$ refer to the numbers in \code{ReflectionClasses}. So, if $s$ is a reflection of $\Gamma$ and $i$ is its reflection class number, then $c(i) = c(s)$. 

The command \code{CherednikParameter} has the additional option \code{Type} which allows specification of different types of parameters. In the above, we selected the GGOR type (see \S\ref{martino_prep}). We can instead also pass \code{EG} as type which are the parameters used in \cite{EG-Symplectic-reflection-algebras} or we can pass \code{BR} which are the parameters used in \cite{Bonnafe.C;Rouquier.R13Cellules-de-Calogero}. There is a further option \code{Rational} which, when set to false, returns the parameter with values in the polynomial ring instead of the rational function field. Instead of using generic parameters, the user can define any map $c:\lbrack 1, N \rbrack \rarr L$ as above which can be used for a Cherednik parameter. \\
Rational Cherednik algebras can be created as follows:
\begin{lstlisting}
> G:=ExceptionalComplexReflectionGroup(4); 
> c:=CherednikParameter(G : Type:="EG");
> H:=RationalCherednikAlgebra(G,<1,c>); H;
Rational Cherednik algebra
Generators:
    g1, g2, y1, y2, x1, x2
Generator degrees:
    0, 0, -1, -1, 1, 1
Base ring:
    Multivariate rational function field of rank 2 over Cyclotomic Field 
    of order 3 and degree 2
    Variables: k_{1,1}, k_{1,2}
Group:
    MatrixGroup(2, Cyclotomic Field of order 3 and degree 2) of order 
    2^3 * 3
    Generators:
    [     1      0]
    [     0 zeta_3]

    [1/3*(2*zeta_3 + 1) 1/3*(2*zeta_3 - 2)]
    [  1/3*(zeta_3 - 1)   1/3*(zeta_3 + 2)]
t-parameter:
    1
c-parameter:
    Mapping from: { 1 .. 2 } to Multivariate rational function field of 
    rank 2 over Cyclotomic Field of order 3 and degree 2
    <1, (-zeta_3 + 1)*k_{1,1} + (2*zeta_3 + 1)*k_{1,2}>
    <2, (zeta_3 + 2)*k_{1,1} + (-2*zeta_3 - 1)*k_{1,2}>
> H.3*H.5;
[1 0]
[0 1]*(y1*x1)
> H.5*H.3;
[1/3*(-2*zeta_3 - 1) 1/3*(-2*zeta_3 - 4)]
[  1/3*(-zeta_3 - 2)   1/3*(-zeta_3 + 1)]*(1/3*(2*zeta_3 + 4)*k_{1,1} + 
1/3*(-4*zeta_3 - 2)*k_{1,2})
+
[1/3*(-2*zeta_3 - 1) 1/3*(-2*zeta_3 + 2)]
[ 1/3*(2*zeta_3 + 1)   1/3*(-zeta_3 + 1)]*(1/3*(2*zeta_3 + 4)*k_{1,1} + 
1/3*(-4*zeta_3 - 2)*k_{1,2})
+
[1/3*(-2*zeta_3 - 1)  1/3*(4*zeta_3 + 2)]
[  1/3*(-zeta_3 + 1)   1/3*(-zeta_3 + 1)]*(1/3*(2*zeta_3 + 4)*k_{1,1} + 
1/3*(-4*zeta_3 - 2)*k_{1,2})
+
[1 0]
[0 1]*(y1*x1 + 1)
+
[ 1/3*(2*zeta_3 + 1)  1/3*(2*zeta_3 + 4)]
[1/3*(-2*zeta_3 - 1)    1/3*(zeta_3 + 2)]*(1/3*(-2*zeta_3 + 2)*k_{1,1} + 
1/3*(4*zeta_3 + 2)*k_{1,2})
+
[ 1/3*(2*zeta_3 + 1) 1/3*(-4*zeta_3 - 2)]
[   1/3*(zeta_3 + 2)    1/3*(zeta_3 + 2)]*(1/3*(-2*zeta_3 + 2)*k_{1,1} + 
1/3*(4*zeta_3 + 2)*k_{1,2})
+
[1/3*(2*zeta_3 + 1) 1/3*(2*zeta_3 - 2)]
[  1/3*(zeta_3 - 1)   1/3*(zeta_3 + 2)]*(1/3*(-2*zeta_3 + 2)*k_{1,1} + 
1/3*(4*zeta_3 + 2)*k_{1,2})
\end{lstlisting}
In the above example we created the opposite rational Cherednik algebra $H^{\mrm{op}} \dopgleich \rH_{1,\bic}^{\mrm{op}}$ for $\rG_4$ and the rational point $\bic$ of $\fR_\Gamma$. The generators of $H$ can be accessed via \code{H.i}, where $i$ lies between $2d+e$, where $d$ is the dimension of $\Gamma$ and $e$ is the number of generators of $\Gamma$. We see in the above output that the generators are ordered as $g_1,g_2,y_1,y_2,x_1,x_2$. In PBW basis expressions group algebra elements are always on the left and in matrix form. In the example we computed the products $y_1x_1$ and $x_1y_1$. Keep in mind that $H$ as created is the \textit{opposite} algebra to what we treated theoretically before. This is why $x_1y_1$ is \textit{not} in PBW form—it is actually the product $y_1x_1$, and this has to be rewritten.

\begin{example}
The following very elaborate example from Bonnafé–Rouquier \cite[\S19]{Bonnafe.C;Rouquier.R13Cellules-de-Calogero} can be treated easily in \champ. The Weyl group of type $\rB_2$ can be realized as the matrix group $\Gamma$ in $\GL_2(\bbQ)$ generated by the reflections
\[
s \dopgleich g_1 \dopgleich \begin{pmatrix} 0 & 1 \\ 1 & 0 \end{pmatrix}, \quad t \dopgleich g_2 \dopgleich \begin{pmatrix} -1 & 0 \\ 0 & 1 \end{pmatrix} \;.
\]
Let $y_1,y_2$ be the standard basis of $V \dopgleich \bbQ^2$ and let $x_1,x_2$ be the dual basis. Let $\lbrace A,B \rbrace$ be algebraically independent over $\bbQ$ and define $\bic_s \dopgleich -2A$, $\bic_{t} \dopgleich -2B$. As $s$ and $t$ are representatives of the conjugacy classes of reflections of $\Gamma$, this yields a map $\bic:\sC_\Gamma \rarr \bbQ(A,B)$ giving the generic point of $\fR_\Gamma$. Now, define the following elements of $\rH_{0,\bic}$:
\[
\sigma \dopgleich y_1^2+y_2^2\;, \quad \pi \dopgleich y_1^2y_2^2\;, \quad \Sigma \dopgleich x_1^2 + x_2^2\;, \quad \Pi \dopgleich x_1^2x_2^2 \;.
\]
In \cite[19.4.5]{Bonnafe.C;Rouquier.R13Cellules-de-Calogero} it is now proven that the Euler element $\mrm{eu}_\bic \in \rH_{0,\bic}$ is a zero of the polynomial
\begin{align*} \label{eu_B2_pol}
& t^8 - 2(\sigma \Sigma + 4A^2 + 4B^2)t^6 \\
& + (\sigma^2\Sigma^2 + 2(\sigma^2\Pi + \Sigma^2\pi - 8\pi\Pi) + 8(A^2 + B^2)\sigma\Sigma + 16(A^2-B^2)^2)t^4 \\
& - 2((\sigma\Sigma + 4A^2-4B^2)(\sigma^2\Pi+\Sigma^2\pi)-8\sigma\Sigma\pi\Pi+2B^2\sigma^2\Sigma^2)t^2 + (\sigma^2\Pi - \Sigma^2\pi)^2 \;.
\end{align*}
This fact was one essential part in determining the Calogero–Moser cells and to prove the Calogero–Moser cell conjecture for $\rB_2$. In \cite{Bonnafe.C;Rouquier.R13Cellules-de-Calogero} this is proven by an argument based on the undeformed situation in $\rH_{0,0}$. As the computation is quite elaborate and one does not want to write down all its details, let us see if we can verify this fact with \champ:
\begin{lstlisting}
> G:=CHAMP_GetFromDB("GrpMat/B2_BR","GrpMat"); //loads B2 as above
> C:=CherednikParameter(G:Type:="BR"); 
> H:=RationalCherednikAlgebra(G,C); 
> eu:=EulerElement(H); eu;
[1 0]
[0 1]*(y1*x1 + y2*x2)
+
[0 1]
[1 0]*(-C1)
+
[-1  0]
[ 0  1]*(-C2)
+
[ 1  0]
[ 0 -1]*(-C2)
+
[ 0 -1]
[-1  0]*(-C1)
> A:=C(1)*(-1/2); B:=C(2)*(-1/2);
> y2:=H.4; y1:=H.3; g2:=H.2; g1:=H.1; x2:=H.6; x1:=H.5;
> sigma:=y1^2+y2^2; pi:=y2^2*y1^2; Sigma:=x1^2+x2^2; Pi:=x2^2*x1^2;
> time eu^8 - 2*eu^6*(Sigma*sigma + 4*A^2 + 4*B^2) + 
eu^4*(Sigma^2*sigma^2 + 2*(Pi*sigma^2 + pi*Sigma^2 - 8*Pi*pi) + 
8*Sigma*sigma*(A^2+B^2) + 16*(A^2-B^2)^2) - 
2*eu^2*( (Pi*sigma^2 + pi*Sigma^2)*(Sigma*sigma + 
4*A^2 - 4*B^2) - 8*Pi*pi*Sigma*sigma + Sigma^2*sigma^2*B^2*2) + 
(Pi*sigma^2 - pi*Sigma^2)^2;
0
Time: 2.360
\end{lstlisting}
Hence, we could indeed verify (within only 2 seconds) that the Euler element is a zero of the polynomial above. Note again that we reversed all products as \champ\ works in the opposite algebra.
\end{example}

\subsection{Verma modules} \label{verma_example}

Let us now see how we can compute in \champ\ with Verma modules for restricted rational Cherednik algebras and how we can answer Gordon's questions:
\begin{lstlisting}
> G:=ExceptionalComplexReflectionGroup(4); Representations(~G,0);
> c:=CherednikParameter(G:Rational:=false); c;
Mapping from: { 1 .. 2 } to Polynomial ring of rank 2 over Cyclotomic 
Field of order 3 and degree 2
    <1, (-zeta_3 + 1)*k_{1,1} + (2*zeta_3 + 1)*k_{1,2}>
    <2, (zeta_3 + 2)*k_{1,1} + (-2*zeta_3 - 1)*k_{1,2}>
> R:=Codomain(c); R;
Polynomial ring of rank 2 over Cyclotomic Field of order 3 and degree 2
Order: Lexicographical
Variables: k_{1,1}, k_{1,2}
> cH:=SpecializeCherednikParameterInHyperplane(c, R.1-R.2); c;
Mapping from: { 1 .. 2 } to Multivariate rational function field of 
rank 1 over Cyclotomic Field of order 3 and degree 2
    <1, (zeta_3 + 2)*k_{1,2}>
    <2, (-zeta_3 + 1)*k_{1,2}>
> EulerFamilies(G,cH);
{@
    <{@ 5, 6 @}, 2*k_{1,2}>,
    <{@ 7 @}, 0>,
    <{@ 2, 3, 4 @}, -4*k_{1,2}>,
    <{@ 1 @}, 8*k_{1,2}>
@}
> V:=VermaModule(G,cH,G`Representations[0][2]); V;
Graded module of dimension 24 over an algebra with generator degrees 
[ -1, -1, 0, 0, 1, 1 ] over Multivariate rational function field of 
rank 1 over Cyclotomic Field of order 3 and degree 2.
> res, P, dims, Pseries, D, Gstruct, L := Gordon(G,cH, [ 2,3,4 ] : 
GeneratorSets:=[{1,2,3}], pExclude:={2,3,5});
> P;
<[ 735 ], Prime Ideal
Two element generators:
    [1873, 0]
    [115, 1]>
> dims;
[ 9, 1, 7 ]
> Pseries;
[
    1 + 2*t + 3*t^2 + 2*t^3 + t^4,
    1,
    2 + 3*t + 2*t^2
]
> D;
[1 1 2]
[1 1 2]
[2 2 4]
> Gstruct;
[*
    (    t^4       1       0       0       0 t^3 + t     t^2),
    (0 0 1 0 0 0 0),
    (  0   0   0   1 t^2   0   t)
*]
> L;
[*
    Graded module of dimension 9 over an algebra with generator degrees 
    [ -1, -1, 0, 0, 1, 1 ] over Multivariate rational function field of 
    rank 1 over Cyclotomic Field of order 3 and degree 2.,
    Graded module of dimension 1 over an algebra with generator degrees 
    [ -1, -1, 0, 0, 1, 1 ] over Multivariate rational function field of 
    rank 1 over Cyclotomic Field of order 3 and degree 2.,
    Graded module of dimension 7 over an algebra with generator degrees
    [ -1, -1, 0, 0, 1, 1 ] over Multivariate rational function field of 
    rank 1 over Cyclotomic Field of order 3 and degree 2.
*]
> IsModuleForRRCA(G,cH,L[1]);
true
\end{lstlisting}
In this example we are considering the group $\rG_4$. At the beginning we create the (non-rational) generic Cherednik parameter of GGOR type. We specialize this parameter in the hyperplane $H$ defined by $k_{1,1}-k_{1,2}$ of $\fR_\Gamma$ in GGOR parameters and get in this way the generic point $\bic_H$ of this hyperplane. We then compute the Verma module $\Delta_{\bic_H}(\phi_{1,4})$ which is a graded module of type \code{ModGr}. The central command is now \code{Gordon} which takes as input a reflection group $G$, a Cherednik parameter, and a list of integers referring to the irreducible representations of $G$ as in the attribute \code{G`Representations}. In the above example we apply it to the Euler $\bic_H$-family $\lbrace \phi_{1,4}, \phi_{1,8}, \phi_{2,5} \rbrace$. This command computes the corresponding Verma modules and applies our algorithms (encapsulated in the command \code{HeadsOfLocalModules}) to compute their heads and their decompositions. The additional option \code{GeneratorSets} controls which generators are used for the \textsc{ModFinder} algorithm (we chose in this case $y_1,y_2,g_2$ as generators) and the option \code{pExclude} describes the primes to be excluded when picking a finite field specialization (in this case we chose $2$, $5$, and $7$ as they are bad). We remark that many additional techniques on which we cannot comment here are ``secretly'' applied while running this command (see also \S\ref{experiments}). If successful, the output consists of the parameters and the prime ideal chosen for the finite field specialization, the dimensions of the simple modules, their Poincaré series, the decomposition matrix of the Verma modules (the entry $(i,j)$ in this matrix is the multiplicity of the head of the $j$-th Verma module in the $i$-th Verma module in the list passed to \code{Gordon}), the graded $G$-module structure of the simple modules and the simple modules themselves. Using the command \code{IsModuleForRRCA} we can check if a family of matrices indeed defines a module for the restricted rational Cherednik algebra—all necessary relations are checked. 

This example is the prototype showing how we can answer all of Gordon's questions by simply applying the command \code{Gordon} to Euler families. 

\subsection{Database}

 \label{champ_db}
All results we could compute so far are contained in an easily accessible database as illustrated by the following example:
\begin{lstlisting}
> G:=ExceptionalComplexReflectionGroup(4);                       
> answers:=Gordon(G);
> answers;
Associative Array with index universe Polynomial ring of rank 2 over 
Cyclotomic Field of order 3 and degree 2
> Keys(answers);
{
    k_{1,2},
    k_{1,1} - 2*k_{1,2},
    k_{1,1},
    2*k_{1,1} - k_{1,2},
    1,
    k_{1,1} + k_{1,2},
    k_{1,1} - k_{1,2}
}
> P:=Universe(Keys(answers)); answers[P.1-P.2];  
rec<recformat<Hyperplane, EulerFamilies, SimpleDims, SimplePSeries,
SimpleGModStruct, SimpleGradedGModStruct, VermaDecomposition,
CMFamilies> | 
    Hyperplane := k_{1,1} - k_{1,2},
    EulerFamilies := {
        { 1 },
        { 2, 3, 4 },
        { 7 },
        { 5, 6 }
    },
    SimpleDims := [ 24, 9, 1, 7, 8, 16, 24 ],
    SimplePSeries := [
        1 + 2*t + 3*t^2 + 4*t^3 + 4*t^4 + 4*t^5 + 3*t^6 + 2*t^7 + t^8,
        1 + 2*t + 3*t^2 + 2*t^3 + t^4,
        1,
        2 + 3*t + 2*t^2,
        2 + 4*t + 2*t^2,
        2 + 4*t + 4*t^2 + 4*t^3 + 2*t^4,
        3 + 6*t + 6*t^2 + 6*t^3 + 3*t^4
    ],
    SimpleGModStruct := [
        (1 1 1 2 2 2 3),
        (1 1 0 0 0 2 1),
        (0 0 1 0 0 0 0),
        (0 0 0 1 1 0 1),
        (0 0 1 1 1 0 1),
        (1 1 0 1 1 2 2),
        (1 1 1 2 2 2 3)
    ],
    SimpleGradedGModStruct := [
        (              1             t^8             t^4       
        t^5 + t^7         t + t^3       t^3 + t^5 t^2 + t^4 + t^6),
        (    t^4       1       0       0       0 t + t^3     t^2),
        (0 0 1 0 0 0 0),
        (  0   0   0   1 t^2   0   t),
        (  0   0   t t^2   1   0   t),
        (      t     t^3       0     t^2     t^2 1 + t^4 t + t^3),
        (          t^2           t^2           t^2       t + t^3       
        t + t^3       t + t^3 1 + t^2 + t^4)
    ],
    VermaDecomposition := [
        (1 0 0 0 0 0 0),
        (0 1 1 2 0 0 0),
        (0 1 1 2 0 0 0),
        (0 2 2 4 0 0 0),
        (0 0 0 0 2 2 0),
        (0 0 0 0 2 2 0),
        (0 0 0 0 0 0 3)
    ],
    CMFamilies := {
        { 1 },
        { 2, 3, 4 },
        { 7 },
        { 5, 6 }
    }>
> RouquierFamilies(G)[P.1-P.2];
{
    { 1 },
    { 2, 3, 4 },
    { 7 },
    { 5, 6 }
}
\end{lstlisting}
In this example we fetched all the results for the example discussed in \S\ref{verma_example} from the database. This is done by calling the command \code{Gordon} for an exceptional complex reflection group created with \code{ExceptionalComplexReflectionGroup}. The result is an associative array indexed by \textit{normalized} equations for the hyperplanes of the Euler variety, and by $1$ signifying generic parameters. It is now easy to test conjectures on these results without performing any additional computations.

 \begin{remark}
 As we want to ensure verifiability of our results we have included in the directory \code{Experiments/GordonQuestions} in \textsc{Champ} scripts which allow a re-computation from scratch of all our results and show how exactly we computed them.
 \end{remark}
 
 \section{Experimental aspects} \label{experiments}
 
 \begin{parani}
The run time and success of the \textsc{ModFinder} algorithm can depend heavily on the input data and on the choices made. We therefore point out some issues we observed in experiments with the hope that future developments will clarify these aspects and lead to further improvements.
 \end{parani}

\subsection{The effect of the choice of generators and realizations}

In Table \ref{verma_head_computation_data} we list some data concerning the computation of the Verma modules and the head of a Verma modules using our algorithm. All computations and time measurements have been performed on an Intel\textregistered\ Core\texttrademark\ i7-3930K $@$ 3.2GHz running the AVX version of \textsc{Magma} 2.19-8.  We always work with generic GGOR parameters and use the realizations of the exceptional complex reflections groups and their representations as obtained from \textsc{Chevie} (these are also the ones used in \champ\ by default).

The columns denoted by $t_\Delta$ give the time needed for computing the $X$-table explained in \S\ref{X_tables} and the time it then takes to compute the corresponding Verma module. The column \textit{Vars} lists the number of variables in the abstract structure of the Jacobson radical of the Verma module (note that our algorithm has to be successful to determine this number). The column $\mbi{g}$ lists the generators we have selected for the \textsc{ModFinder} algorithm. In the last columns denoted by $t_{\Hd \Delta}$ we list the time the \textsc{MeatAxe} needed to determine the Jacobson radical of the finite field specialization of the Verma module, the time the \textsc{ModFinder} needed and the total time (this includes for example the graded spinning algorithm to ensure that we found a submodule).   
     \begin{table}[htbp] \label{computations_data} \footnotesize
   \centering
    \begin{tabular}{|c|c|c|c|c|c|c|c|c|c|c|}
      \hline
      $G$ & $\lambda$ & $\dim \Delta$ & \multicolumn{2}{c|}{$t_{\Delta}$} & $\mrm{dim} \Hd \Delta$ & Vars & $\mbi{g}$ & \multicolumn{3}{c|}{$t_{\Hd \Delta}$}  \\ \hline \hline
      $\rG_4$ & $\phi_{3,2}$ & $72$ & $0.19$ & $0.21$ & $24$ & $52$ & $\lbrace y_2  \rbrace$  & $0.01$ & $0.73$ & $1.98$ \\ \hline
      $\rG_5$ & $\phi_{3,6}$ & $216$  & $2.19$ & $1.78$ & $24$ & $70$ & $ \lbrace y_2 \rbrace $ & $0.12$ & $52.61$ & $74.58$ \\ \hline
       $\rG_5$ & $\phi_{3,6}$ & .  & . & . & . & . & $ \lbrace y_2,x_2 \rbrace $ & . & $12.06$ & $33.94$ \\ \hline
       $\rG_5$ & $\phi_{3,6}^{(1)}$ & .  & . & $2.4$ & . & $24$ & $ \lbrace y_2 \rbrace $ & $0.12$ & $0.67$ & $1.83$ \\ \hline
       $\rG_7$ & $\phi_{2,15}$ & $288$  & $10.39$ & $5.73$ & $72$ & $208$ & $ \lbrace y_2, y_1 \rbrace $ & $0.22$ & $735.70$ & $860.56$ \\ \hline
       $\rG_7$ & $\phi_{2,15}$ & $.$  & $.$ & $.$ & $.$ & $.$ & $ \lbrace y_2, g_1 \rbrace $ & . & $205.01$ & $329.81$ \\ \hline
       $\rG_{23}$ & $\phi_{4,4}$ & $480$ & $12.15$ & $10.27$ & $60$ & $759$ & $\lbrace y_3,y_2 \rbrace$ & $3.74$ & $40.49$ & $72.04$ \\ \hline
       $\rG_9$ & $\phi_{3,4}$ & $576$ & $23.56$ & $11.16$ & $192$ & $491$ & $\lbrace y_2 \rbrace$ & ? & ? & ?\\ \hline
       $\rG_9$ & $\phi_{3,4}^{(2)}$ & . & . & $36.72$& . & $90$ & $\lbrace y_2 \rbrace$ & $1.13$ & $4.65$ & $11.83$ \\ \hline 
       $\rG_9$ & $\phi_{3,4}^{(1)}$ & . & . & $12.66$ & . & $630$ & $\lbrace y_2 \rbrace$ & ? & ? & ? \\ \hline
       $\rG_{24}$ & $\phi_{3,8}^{(1)}$ & $1008$ & $206.43$ & $91.85$ & $156$ & $3888$ & $\lbrace y_3,y_2 \rbrace$ & $24.22$ & $595.72$ & $849.51$\\ \hline
       $\rG_{24}$ & $\phi_{3,10}^{(1)}$ & $1008$ & . & $99.50$ & $6$ & $14$ & $\lbrace y_3,y_2 \rbrace$ & $28.47$ & $0.13$ & $50.17$ \\ \hline
    \end{tabular} 
    \caption{Experimental data about the computation of the heads of Verma modules.} \label{verma_head_computation_data}  
    \end{table} 
This table shows us immediately how sensitive our approach is to the choices we make throughout. Let us discuss this in more detail. \\

First of all, we can see that we usually work with very small $\mbi{g}$. We almost never had to consider all algebra generators for the \textsc{ModFinder} algorithm. For the computation of the head of the Verma module $\Delta_\bik(\phi_{3,6})$ for $\rG_5$, however, we see that the selection of $\mbi{g}$ can be important. It this situation the choice $\mbi{g} = \lbrace y_2,x_2 \rbrace$ is more than twice as fast as $\lbrace y_2 \rbrace$. Unfortunately we cannot say yet what makes one choice better than the other—we just found efficient choices by experimenting and it seems best to start with the basis $(y_i)_{i=1}^n$ of $V$.

Next, we observed that when modifying our explicit realizations of the group and the irreducible representations of the group in such a way that one generator of the group is diagonal and acts diagonally on all representations the \textsc{ModFinder} algorithm usually performs much faster. We denote in this table by $\lambda^{(i)}$ the representation obtained from $\lambda$ by changing the basis so that the generator $i$ of the chosen realization of the group acts diagonally. Comparing the computations for $\phi_{3,6}$ and $\phi_{3,6}^{(1)}$ for $\rG_5$ we see that we obtained the solution for $\phi_{3,6}^{(1)}$ around 20 times faster than for $\phi_{3,6}$. We see that the number of variables in the Jacobson radical drops from $70$ to only $24$ which is probably the reason for the speedup. Because of this the command \code{Gordon} always automatically performs such a diagonalization, respecting the fact that the realizations of the exceptional complex reflection groups in \textsc{Chevie} are usually chosen such that one generator is already diagonal.

In the example $\phi_{3,8}^{(1)}$ for $\rG_{24}$ we see that even a very large number of variables ($3888$ in this case) do not necessarily have to be a problem. We are able to compute the head of the corresponding $1008$-dimensional Verma module in just around 15 minutes. Even more fascinating is the example $\phi_{3,10}^{(1)}$ for $\rG_{24}$. Here, we finish the determination of the $1002$-dimensional Jacobson radical in just $50$ seconds (the \textsc{ModFinder} algorithm just needs $0.13$ seconds).

We see from these examples that our algorithm can be surprisingly powerful but that it is very hard to control theoretically.

\subsection{Comparison with the algorithm in Magma} \label{magmas_algo_timing}
 So far we did not comment on other already existing algorithms to compute the heads of the Verma modules in characteristic zero. The \textsc{MeatAxe} might actually solve this problem in special situations. In his PhD thesis Steel \cite{Steel.A12Construction-of-Ordi} has developed a general characteristic zero \textsc{MeatAxe} which is in theory capable of computing the radical of a module over an algebra over a field of characteristic zero. This algorithm is implemented in \textsc{Magma} since 2012 and it is---to our knowledge---the only algorithm which could also be used to compute the head of Verma modules for restricted rational Cherednik algebras.\footnote{Unfortunately, it seems that there is no publication describing these methods in detail. Nevertheless, we argue here that \textit{in our case at least} we have more significant reasons for not using it.} We therefore have to compare our methods with this algorithm. As it is also a Las Vegas algorithm, we cannot simply test it once for a specific problem and record the run time because it might always be the case that the randomly chosen parameters were bad. We thus have to run several tests and determine the average run time. We run each attempt with a time out $\tau$ of 900 seconds (15 minutes) for each attempt as the run time of our algorithm is always much lower. We then record the average run time of all successful approaches, and record the success rate $\alpha$ within the time window $\tau$ for specific problems. The results are listed in Table \ref{magma_algo_comparison_table}.
\begin{table}[htbp] \footnotesize
    {\centering
    \begin{tabular}{|c|c|c||c|c||c|c|}
      \hline
    $\rG$ & $\lambda$ & Tests & \textsc{Magma} Avg. & \textsc{Magma} $\alpha$ & \textsc{Champ} Avg. & \textsc{Champ} $\alpha$ \\ \hline \hline
    $\rS_4$ & $(2,1,1)$ & $82$ & $0.65$ & $0.13$ & $0.23$ & $1.0$ \\ \hline
    $\rG_4$ & $\phi_{3,7}^{(1)}$ & $84$ & $0.76$ & $0.15$ & $0.7$ & $1.0$ \\ \hline
    $\rG_4$ & $\phi_{3,7}$ & $82$ & -- & $0.0$ & $5.2$ & $1.0$ \\ \hline
    $\rG_{12}$ & $\phi_{4,3}^{(3)}$ & $84$ & $3.29$ & $0.14$ & $0.38$ & $1.0$ \\ \hline
    $\rG_{6}$ & $(\phi_{2,5}')^{(2)}$ & $77$ & -- & $0.0$ & $0.25$ & $1.0$ \\ \hline
    $\rG_6$ & $(\phi_{2,3}'')^{(2)}$ & $79$ & -- & $0.0$ & $0.25$ & $1.0$ \\ \hline
    $\rG_5$ & $\phi_{3,6}^{(1)}$ & $81$ & -- & $0.0$ & $5.1$ & $1.0$ \\ \hline
    $\rG_7$ & $(\phi_{2,11}')^{(1)}$ & $78$ & -- & $0.0$ & $42.0$ & $1.0$ \\ \hline
\end{tabular}

}
\caption{Comparison of \textsc{Magma}'s algorithm (left) with ours (right).} \label{magma_algo_comparison_table}
\end{table}

We see that our success rate is always 100\% while \textsc{Magma}'s success rate is below 15\%—if there is success at all. For all problems where \textsc{Magma}'s algorithm did not return a result within the time window $\tau$ we also did not get a result in sporadic attempts after a couple of days. Although this does not mean that \textsc{Magma}'s algorithm would not eventually solve the problem, it should be quite clear from the table that without our algorithm we would not have been able to obtain most results in \S\ref{further_results}---in particular since the modules we have to work with are much bigger than those listed in the table. 

\begin{remark}
As our algorithm for determining the head of a module with simple head in characteristic zero is completely general (despite non-trivial theoretical assumptions which have to be checked in each case), it is in principle applicable to many more situations. We hope that future developments and improvements to this method will enable us to solve problems in other contexts.  
\end{remark}

\end{document}